\documentclass[onefignum,onetabnum]{siamart190516}

\usepackage{lipsum}
\usepackage{amsfonts}
\usepackage{bbm}
\usepackage{graphicx}
\usepackage{epstopdf}
\usepackage{booktabs}
\usepackage{amsmath}
\usepackage{amssymb}
\usepackage{cite}
\usepackage{tikz,pgfplots}
\usetikzlibrary{trees,calc}
\usepackage{tikz-dimline}
\usepackage{accents}
\usepackage{algorithmic}
\ifpdf
  \DeclareGraphicsExtensions{.eps,.pdf,.png,.jpg}
\else
  \DeclareGraphicsExtensions{.eps}
\fi

\newsiamremark{remark}{Remark}
\newsiamremark{hypothesis}{Hypothesis}
\crefname{hypothesis}{Hypothesis}{Hypotheses}
\newsiamthm{claim}{Claim}
\newsiamthm{assumption}{Assumptions}

\headers{AL preconditioning for variable viscosity Stokes}{}

\title{Robust multigrid techniques for augmented Lagrangian preconditioning of
  incompressible Stokes equations with extreme viscosity
  variations\thanks{Submitted to the editors DATE.
    \funding{This work was partially supported by the US National Science
      Foundation (NSF) through grant EAR \#1646337, and by the SciDAC
      program funded by the U.S.\ Department of Energy, Office of
      Science, Advanced Scientific Computing Research, and Biological
      and Environmental Research Programs, and a grant from the
Simons Foundation (560651).}}}

 \author{Yu-hsuan Shih\thanks{Courant Institute, New York
       University, New York, USA (\email{shihyh@cims.nyu.edu},
       \email{wechsung@cims.nyu.edu}, \email{stadler@cims.nyu.edu}).}
     \and Georg Stadler$^\dagger$ \and Florian Wechsung$^\dagger$}

\usepackage{amsopn}

\newcommand{\boldvar}[1]{\ensuremath{\boldsymbol{#1}}}
\newcommand{\gradient}{\nabla}
\newcommand{\divergence}{\nabla\cdot}
\newcommand{\vel}{\boldvar{u}}
\newcommand{\veltest}{\boldvar{v}}
\newcommand{\press}{p}
\newcommand{\presstest}{q}
\newcommand{\visc}{\mu}
\newcommand{\trans}{\mathsf{T}}
\newcommand{\transpose}{^{\trans}}
\newcommand{\rhsvel}{\boldvar{f}}

\newcommand{\strainrate}{\dot{\varepsilon}}
\newcommand{\strainratetensor}{\dot{\boldvar{\varepsilon}}}

\newcommand{\viscstress}{\tau}
\newcommand{\viscstresstensor}{\boldvar{\tau}}
\newcommand{\invII}[1]{#1_\textsc{ii}} 
\newcommand{\strainrateinvII}{\invII{\strainrate}}
\newcommand{\viscstressinvII}{\invII{\viscstress}}
\newcommand{\viscstressyield}{\viscstress_{\mathrm{y}}}
\newcommand{\viscref}{{\visc_{\mathrm{r}}}}
\newcommand{\viscmin}{\underline\visc}

\allowdisplaybreaks

\newlength{\dhatheight}

\begin{document}

\maketitle

\begin{abstract}
We present augmented Lagrangian Schur complement preconditioners and
robust multigrid methods for incompressible Stokes problems with
extreme viscosity variations. Such Stokes systems arise, for instance,
upon linearization of nonlinear viscous flow problems, and they can
have severely inhomogeneous and anisotropic coefficients. Using an
augmented Lagrangian formulation for the incompressibility constraint
makes the Schur complement easier to approximate, but results in a
nearly singular (1,1)-block in the Stokes system. We present
eigenvalue estimates for the quality of the Schur complement
approximation.
To cope with the near-singularity of the
(1,1)-block, we
extend a multigrid scheme with a discretization-dependent
smoother and transfer operators from triangular/tetrahedral to the 
quadrilateral/hexahedral
finite element discretizations $[\mathbb{Q}_k]^d\times \mathbb{P}_{k-1}^{\text{disc}}$, $k\geq 2$, $d=2,3$.
Using numerical examples with scalar
and with anisotropic fourth-order tensor viscosity arising
from linearization of a viscoplastic constitutive relation, we confirm
the robustness of the multigrid scheme and the overall efficiency of
the solver. We present scalability results using up to
28,672 parallel tasks for problems with up to 1.6 billion unknowns and
a viscosity contrast up to ten orders of magnitude.
\end{abstract}



\section{Introduction}
Viscous flows governed by equations with strongly nonlinear and/or
inhomogeneous rheologies play an important role in
applications.  They are, for instance, used to describe flows in
porous media \cite{Bear13}, the behavior of the solid earth over long
time scales \cite{Ranalli95}, the dynamics of continental ice sheets
and glaciers \cite{Hutter83}, and the phenomenological behavior of
colloidal dispersions \cite{MewisWagner12}. These and other phenomena
can be described by the incompressible Stokes equations on a domain
$\Omega\subset\mathbb R^{d}$, $d=2,3$,
\begin{subequations}
\label{eq:stokes}
\begin{alignat}{2}
  \label{eq:momentum}
  - \divergence
    \bigl[ \visc(\boldvar x, \strainrateinvII) \, (\gradient\vel + \gradient\vel\transpose) \bigr]
  + \gradient\press &= \rhsvel
  &&\quad\text{in }\Omega,
  \\
  \label{eq:mass}
  - \divergence\vel &= 0
  &&\quad\text{in }\Omega,
\end{alignat}
\end{subequations}
where $\vel$ and $\press$ are the velocity and pressure fields and
$\rhsvel$ is a volumetric force. The viscosity
$\visc(\boldvar x, \strainrateinvII)$ may depend explicitly on
$\boldvar x\in \Omega$, but also on the unknown solution, typically on
the second invariant of the strain rate tensor $\strainrateinvII$.
For incompressible velocity $\vel$, $\strainrateinvII$ is given by
$\strainrateinvII := (\frac12\,\strainratetensor(\vel) :
\strainratetensor(\vel))^{1/2}$, where $\strainratetensor(\vel) :=
\frac12 (\gradient\vel + \gradient\vel\transpose)$ is the strain rate
tensor.  The dependence of the viscosity on the solution makes
\eqref{eq:stokes} nonlinear, and thus the solution of
\eqref{eq:stokes} requires linearization. This nonlinearity and/or the
explicit spatial dependence of the viscosity can lead to localized
solution features, e.g., when narrow shear zones occur through strain weakening,
or when geometric features are incorporated through a spatially varying
viscosity. Resolving such localized
features in numerical simulations typically requires (locally) refined
meshes, resulting in large and poorly conditioned (non)linear
systems of equations to be solved. Such systems, which can easily have
tens or hundreds of millions of unknowns, require robust, efficient
and scalable iterative solvers and preconditioners. This paper
presents solvers for linearizations of \eqref{eq:stokes} that result
in Stokes problems with severely inhomogeneous and anisotropic
viscosities.

\subsection{Linearization and discretization}
Linearization of the nonlinear Stokes equations \eqref{eq:stokes} is
typically based on a Picard or a Newton method. The Picard method is a
fixed point iteration that requires solution of a sequence of
linearized Stokes problems with \emph{scalar} viscosity function. It
is well documented in the literature that fixed point methods can
converge slowly, in particular for strongly nonlinear rheologies,
e.g., for problems with viscous-plastic behavior
\cite{SpiegelmanMayWilson16, GrinevichOlshanskii09}. Newton's method
for \eqref{eq:stokes} requires to solve linearized Stokes problems
that involve the sum of a scalar viscosity
and an anisotropic fourth-order tensor viscosity.  This
additional tensor results from linearization of the viscosity with respect
to the velocity due to its dependence on the second invariant of the
strain rate, leading to linearizations of the form
\begin{subequations}
\label{eq:stokes-lin}
\begin{alignat}{2}
  \label{eq:momentum-inc}
  - \divergence
    \Bigl[\Bigl( \visc\boldvar{I} +
      \frac{\partial \visc}{\partial \strainrateinvII}
      \strainratetensor\otimes \strainratetensor \Bigr)
      (\gradient\tilde\vel + \gradient\tilde\vel\transpose) \Bigr]
  + \gradient\tilde \press &= \boldvar r_1
  &&\quad\text{in }\Omega,
  \\
  \label{eq:mass-lin}
  - \divergence\tilde\vel &= r_2
  &&\quad\text{in }\Omega,
\end{alignat}
\end{subequations}
where $\tilde\vel$ and $\tilde\press$ are the velocity and pressure
Newton update variables. Here, the viscosity $\visc$ and strain
rate tensor $\strainratetensor$ are evaluated at the previous velocity
iterate, $\boldvar r_1$ and $r_2$ denote  momentum and mass
equation residuals, and $\otimes$ denotes the outer product between
second-order tensors. Upon discretization, \eqref{eq:stokes-lin}
results in a typical block matrix system of the form
\begin{equation}
\label{eq:stokessys-disc}
\begin{bmatrix}
	\mathbf A & \mathbf B^T\\
	\mathbf B & \mathbf 0
\end{bmatrix}
\begin{bmatrix}
	\tilde {\mathbf u}\\
	\tilde {\mathbf p}
\end{bmatrix}
=
\begin{bmatrix}
	\mathbf r_1\\
	\mathbf r_2
\end{bmatrix},
\end{equation}
where $\mathbf B$ is the discrete divergence operator, and $\mathbf A$
is a discretization of the viscous stress operator.
Even when the viscosity is an anisotropic tensor,
$\mathbf A$ is typically positive definite if reasonable boundary
conditions for the Stokes problem are assumed.  Although the
anisotropic term can degrade the efficiency of iterative solvers,
preconditioners have almost exclusively been studied for scalar
variable viscosity problems \cite{GrinevichOlshanskii09,
  ElmanSilvesterWathen14, RudiStadlerGhattas17, MayMoresi08}. We will illustrate that the solvers we propose
are also able to robustly handle discretizations of \eqref{eq:stokes-lin}
that include anisotropic viscosity.

\subsection{Preconditioning}
The efficiency of iterative Krylov solvers for
\eqref{eq:stokessys-disc} crucially depends on the availability of
effective preconditioners. Arguably, the most popular preconditioners
are based on approximate inversion of the block matrix
\begin{equation}
\label{eq:block-precond}
\begin{bmatrix}
    \mathbf A & \mathbf O\\
    \mathbf O & -\mathbf S
\end{bmatrix} \quad \text{with } \mathbf S := \mathbf B\mathbf A^{-1}\mathbf B^T
\end{equation}
being the Schur complement. Since computing the Schur complement
matrix explicitly is infeasible for large-scale problems, one typically
relies on Schur complement approximations. The most important
approximations are
(weighted) finite element mass matrices
\cite{BursteddeGhattasStadlerEtAl09,GrinevichOlshanskii09,
GeenenurRehmanMacLachlanEtAl09,FuruichiMayTackley11,
ElmanSilvesterWathen14,MayBrownPourhiet15, IsaacStadlerGhattas15} 
and algebraic, so-called BFBT, approximations
\cite{ElmanSilvesterWathen14, RudiStadlerGhattas17, MayMoresi08, RudiMalossiIsaacEtAl15}.
These approximations, and thus the efficiency of the corresponding
preconditioners degrade for very strong viscosity variations or tensor
viscosities as in \eqref{eq:stokes-lin}; see the discussion in
\cref{sec:discreteschur}.

In this paper, we follow the augmented Lagrangian (AL) approach
(see~\cite{FortinGlowinski00, BenziOlshanskii06}), which
replaces the (1,1) block in
\eqref{eq:stokessys-disc} with
$\mathbf A+\gamma \mathbf B^T\mathbf W^{-1} \mathbf B$,
where $\mathbf W^{-1}$ is a positive definite matrix
and $\gamma> 0$. With accordingly modified right hand side of the system 
\eqref{eq:stokessys-disc}, 
this does not change the Stokes solution.
The resulting formulation has the advantage
that its Schur complement is much easier
to approximate for sufficiently large $\gamma$.
However, this simplification comes at the cost
of introducing a term to the (1,1)-block that has a large null space, 
which makes its inversion more difficult.
To invert the (1,1)-block, Benzi and Olshanskii~\cite{BenziOlshanskii06} use a multigrid
algorithm developed by Sch\"oberl~\cite{Schoeberl99} that uses custom smoothing and prolongation
operators and thus does not degrade for large $\gamma$.
However, this multigrid algorithm is highly element specific and hence much of
the subsequent work utilizing AL techniques has either
utilized matrix
factorizations~\cite{deniet2007,urrehman2008,borm2010,he2011,he2012,heister2012}
or block triangular approximations~\cite{benzi2011,hamilton2010,Benzi11b} of
the (1-1)-block.  Recently, in part due to advances in scientific computing
libraries that make the implementation of
advanced multigrid schemes more straightforward~\cite{Kirby18,FarrellKnepleyMitchellWechsung21},
there has been renewed effort to develop and implement robust multigrid schemes
in this context~\cite{FarrellMitchellWechsung19,
FarrellMitchellScottWechsung20a, Xia21,Laakmann2021}.  The discussion of such
methods and their extension to quadrilateral and hexahedral elements is a main
focus of this paper.
We note that augmented Lagrangian preconditioners in the context of variable
viscosity were already studied in~\cite{he2012}.  The differences in our work
are the use of viscosity weighted mass matrices in the Schur complement and the
aforementioned robust multigrid scheme for the (1,1)-block (instead of a direct
solver or algebraic multigrid scheme).  These differences enable us to consider
significantly larger viscosity contrasts and to solve large scale problems in
three dimensions.

An alternative to the above Schur complement-based approaches is to
consider a monolithic method that applies multigrid to the saddle
point system directly. Examples of the associated smoothers applied to
incompressible Stokes equations include
smoothers\cite{BorzacchielloLericheBlottiereGuillet17,
  BraessSarazin97, DrzisgaJohnRudeWohlmuthZulehner18,
  MingLong13}. Stokes problems with variable viscosity are considered
in \cite{BorzacchielloLericheBlottiereGuillet17}, where the authors
show that the robustness of the monolithic multigrid method with
respect to viscosity variation depends on the choice of the
smoother. They propose two Vanka-type smoothers and, for a
test problem, the resulting monolithic multigrid scheme remains
effective up to $10^7$ viscosity contrast.

\subsection{Contributions and limitations}
The main \emph{contributions} in this paper are:
(1) We prove mesh-independent eigenvalue estimates for the Schur
complement approximation of the augmented system in terms of Schur
complement approximations of \eqref{eq:stokes-lin}.
(2) We extend results for parameter-robust multigrid solvers to
element pairings on quadrilateral and hexahedral meshes using novel
arguments to prove the
kernel decomposition property.
(3) We illustrate the efficiency of our preconditioner for linear and
nonlinear problems with up to 10 orders viscosity variation and up to
1.6 
billion unknowns.

The \emph{limitations} of our work are as follows.
%
(1) Our theoretical estimates for the Schur complement approximation use
properties of the Stokes problem and generalization to Navier Stokes
or Oseen problems might not be straightforward.
(2) The parameter-independent smoothers we construct require assembled
stiffness matrices.

\subsection{Notation}
\label{sec:not}
Here, we summarize notation used throughout the paper.  For a
measurable set $G\subset \mathbb{R}^d$, $d=2,3$, we denote by
$(u,v)_{L^2(G)}$ and $\|u\|_{L^2(G)}$ the inner product and the
induced norm in $L^2(G)$, respectively.  When $G=\Omega$, we simply
write $(u,v)$ and $\|u\|_0$.  We use $L^2_0(G)$ to denote the quotient
of $L^2(G)$ with the constant functions, i.e., $L^2_0(G):= \{q\in
L^2(G): (q, 1)_{L^2(G)} =0\}$. For $Q\subset L^2(G)$, we use $\Pi_Q$
to denote the $L^2$-projection operator onto $Q$.
In addition, we denote by
$|u|_{H^1(G)}^2:=(\nabla u, \nabla u)_{L^2(G)}$
and $\|u\|_{H^1(G)}^2:= (u,u)_{L^2(G)}+(\nabla u, \nabla u)_{L^2(G)}$ 
the squared seminorm and norm in the Sobolev space $H^1(G)$,
respectively.  When $G=\Omega$, we simplify the notations to $|u|_1^2$
and $\|u\|_1$.  We denote by $H^1_0(G):=\{u\in H^1(G): u=0 \text{ on
}\partial G\}$ the subspace of $H^1(G)$ containing $H^1(G)$
function that satisfies homogeneous Dirichlet boundary conditions.
Additionally, we use the following notation in estimates.
For $\mathbf A$, $\mathbf B \in \mathbb{R}^{n\times n}$ 
being two symmetric positive definite matrices, $\mathbf A \leq \mathbf B$ means 
that $\mathbf x^T\mathbf A\mathbf x \leq \mathbf x^T\mathbf B \mathbf x$
for all $\mathbf x \in \mathbb{R}^{n}$;
For PDE-discretization matrices  $\mathbf A$, $\mathbf B \in
\mathbb{R}^{n\times n}$,
$\mathbf A \preceq \mathbf B$ means that there exist a mesh-independent constant 
$c$ such that $\mathbf A \leq c\mathbf B$. 
The same notation is also used for scalars derived from discretization
matrices, i.e., $a\preceq b$ means that there is
a mesh-independent constant $c$ such that $a\leq cb$.

\section{Discretization and Schur complement preconditioning}
\label{sec:discreteschur}
The main focus of this paper is on the linearized Stokes problem
\eqref{eq:stokes-lin}. For the analysis presented in the next
sections, we use homogeneous Dirichlet boundary conditions and
consider a problem with scalar viscosity field $\visc(\boldvar
x)\in\mathbb{R}$, which only depends on the spatial variable $\boldvar
x$. However, throughout the remainder of this paper, we comment on practical
aspects when the viscosity is a tensor as in \eqref{eq:stokes-lin},
and present numerical results with anisotropic fourth-order tensor
viscosities in \cref{sec:tensor}.
For simplicity of notation, in the following we use
$\vel,\press$ instead of $\tilde\vel, \tilde\press$, resulting in
\begin{subequations}
\label{eq:stokes-linear}
\begin{alignat}{2}
  - \divergence
    \bigl[ 2\visc(\boldvar x) \, \strainratetensor(\vel) \bigr]
  + \gradient\press &= \boldvar r_1
  &&\quad\text{in }\Omega,
  \\
  - \divergence\vel &= r_2
  &&\quad\text{in }\Omega,\\
  \vel &= 0 &&\quad\text{in }\partial\Omega.
\end{alignat}
\end{subequations}
The weak form of \eqref{eq:stokes-linear} is as follows:
given $\boldvar r_1 \in \left(H^{-1}(\Omega)\right)^d$ and $r_2 \in L^2(\Omega)$, find $\vel \in
\left(H^1_0(\Omega)\right)^d$, $d=2,3$, and $\press
\in L^2_0(\Omega)$ such that
\begin{subequations}
\label{eq:weakstokes}
\begin{alignat}{2}
	a(\vel,\veltest) - (\divergence \veltest,\press) &=
        \langle\boldvar r_1,\veltest\rangle
	&&\quad
	\forall \veltest \in (H^1_0(\Omega))^d,\\
	-(\divergence \vel,q) &= (r_2,q) &&\quad \forall q \in L^2_0(\Omega),
\end{alignat}
\end{subequations}
where $a(\vel,\veltest) = (2\visc(\boldvar x)
\strainratetensor(\vel),\strainratetensor(\veltest))$ and
$\strainratetensor(\vel)=\frac12(\gradient\vel + \gradient\vel\transpose)$ 
is the strain rate tensor. 
Choosing finite element spaces $V_h\subset
(H^1_0(\Omega))^d$ and $Q_h\subset L^2_0$ for velocity and pressure, 
respectively, the discrete algebraic
system corresponding to \eqref{eq:weakstokes} becomes
\begin{equation}
\label{eq:stokessys}
\begin{bmatrix}
	\mathbf A & \mathbf B^T\\
	\mathbf B & \mathbf 0
\end{bmatrix}
\begin{bmatrix}
	\mathbf u\\
	\mathbf p
\end{bmatrix}
=
\begin{bmatrix}
	\mathbf r_1\\
	\mathbf r_2
\end{bmatrix},
\end{equation}
where 
$[\mathbf A]_{i,j}=(2\visc(\boldvar
x)\strainratetensor(\phi_i),\strainratetensor(\phi_j))$ is the 
discrete viscous stress operator, 
$[\mathbf B]_{i,j} = -(\psi_i, \divergence \phi_j)$ is the discrete
divergence operator and 
$[\mathbf B^T]_{i,j}=-(\divergence \phi_i, \psi_j)$ is the discrete gradient
operator. Here, we denote the velocity and pressure basis functions by
$\phi_i$ and $\psi_j$, respectively.
A widely used class of preconditioners for saddle point
systems of the form \eqref{eq:stokessys} are based on the block matrix
identity
\begin{equation}
\begin{bmatrix}
	\mathbf A & \mathbf B^T\\
	\mathbf B & \mathbf 0
\end{bmatrix}
=
\begin{bmatrix}
    \mathbf I & \mathbf O\\
    \mathbf B\mathbf A^{-1} & \mathbf I
\end{bmatrix}
\begin{bmatrix}
    \mathbf A & \mathbf O\\
    \mathbf O & -\mathbf S
\end{bmatrix}
\begin{bmatrix}
    \mathbf I & \mathbf A^{-1}\mathbf B^T\\
    \mathbf O & \mathbf I
\end{bmatrix},
\end{equation}where 
$\mathbf S := \mathbf B\mathbf A^{-1}\mathbf B^T$ is the Schur
complement. This identity motivates that
\eqref{eq:stokessys} can be preconditioned by 
\begin{equation}
\label{eq:precond}
\mathbf P =
\begin{bmatrix}
    \mathbf I & -\hat{\mathbf A}^{-1}\mathbf B^T\\
    \mathbf 0 & \mathbf I
\end{bmatrix}
\begin{bmatrix}
	\hat{\mathbf A}^{-1} & \mathbf O\\
	\mathbf O & -\hat{\mathbf S}^{-1}
\end{bmatrix}
\begin{bmatrix}
    \mathbf I & \mathbf 0\\
	-\mathbf B\hat{\mathbf A}^{-1} & \mathbf I
\end{bmatrix},
\end{equation}
with appropriate choices of $\hat{\mathbf A}$ and $\hat{\mathbf S}$ such that
$\hat{\mathbf A}^{-1}\approx \mathbf A^{-1}$ and $\hat{\mathbf S}^{-1} 
\approx \mathbf S^{-1}$.

Hence, the efficiency of preconditioning with $\mathbf P$ relies on
the availability of good approximations of $\mathbf A^{-1}$ and the
inverse Schur complement $\mathbf S^{-1}$. The Schur
complement typically cannot be computed explicitly for large-scale
problems and one must rely on approximations, and different
approximations result in different preconditioning strategies.
One common choice of the Schur complement approximation is to use
the inverse viscosity-weighted pressure mass matrix
 or its diagonalized versions obtained, for instance, by mass lumping
\cite{BursteddeGhattasStadlerEtAl09, GeenenurRehmanMacLachlanEtAl09,
FuruichiMayTackley11, MayBrownPourhiet15, IsaacStadlerGhattas15,
GrinevichOlshanskii09, ElmanSilvesterWathen14}.
The entries of the inverse viscosity-weighted pressure mass matrix $\mathbf
M_p(1/\visc)$ are given by
\begin{equation}\label{eq:inv-visc-mass}
	[\mathbf M_p(1/\visc)]_{i,j} 
	:= \left({\visc}^{-1}
           \psi_i, \psi_j\right).
\end{equation}
Both the pressure mass matrix $\mathbf M_p := \mathbf M_p(1)$ and 
the inverse viscosity-weighted pressure mass matrix are spectrally equivalent
to the Schur complement \cite{GrinevichOlshanskii09}. It is known that $\mathbf
M_p(1/\visc)$ offers an improvement over $\mathbf M_p$ as an approximation of 
the Schur complement when the viscosity is non-constant.  
However, as discussed and demonstrated in \cite{RudiStadlerGhattas17}, for 
applications with extreme viscosity variations, 
$\mathbf M_p(1/\visc)$ becomes a poor approximation of the Schur
complement, which slows down
the convergence of the iterative solvers. Additionally, for problems
in which the 
viscosity includes an anisotropic term, it is unclear how that term
can be incorporated when $\mathbf M_p(1/\visc)$ is used as Schur
complement approximation---the anisotropic part of the viscosity is thus typically
dropped and only the isotropic component used.

BFBT approximations for the Schur complement, also known as least-squares 
commutators \cite{ElmanSilvesterWathen14}, have also been considered. 
The approximations are of the form 
\begin{equation}
	\hat{\mathbf S}^{-1}_{\text{BFBT}} = (\mathbf B\mathbf C^{-1}\mathbf B^T)^{-1}
	           (\mathbf B\mathbf C^{-1}\mathbf A\mathbf C^{-1}\mathbf B^T)
	           (\mathbf B\mathbf C^{-1}\mathbf B^T)^{-1},
\end{equation}
with some matrix $\mathbf C$. Such an algebraic approach can be favourable
for problems with non-scalar
viscosity as long as $\mathbf C$ is well-defined.  A common drawback
of BFBT approximations is that adjustments are required to accommodate
Dirichlet boundary conditions \cite{ElmanSilvesterWathen14,RudiStadlerGhattas17}, which
increases the complexity of the implementation.  Not surprisingly, the
quality of the approximation depends on the matrix $\mathbf
C$. It has been shown that with appropriate choice of $\mathbf C$, in
particular, $\text{diag}(\mathbf A)$ \cite{MayMoresi08} and
$\tilde{\mathbf M}_{\vel} (\sqrt{\visc(\boldvar x)})$ the lumped
velocity mass matrix weighted by the square root of the viscosity
\cite{RudiStadlerGhattas17}, using $\hat{\mathbf
  S}^{-1}_{\text{BFBT}}$ as the Schur complement approximation leads
to a faster convergence compared to using using $\mathbf M_p(1/\visc)$
as Schur complement approximation for problems with extreme viscosity
variations. Both choices for $\mathbf C$, however, have limitations:
for $\mathbf C=\text{diag}(\mathbf A)$, the effectiveness of the Schur
complement approximation deteriorates with increasing order of the
discretization, $k$, \cite{RudiStadlerGhattas17};
using $\mathbf C=\tilde{\mathbf M}_{\vel} (\sqrt{\visc(\boldvar x)})$
overcomes this limitation and achieves a robust convergence with
respect to the order $k$, but the definition of $\mathbf C$ requires a
scalar viscosity field.

\section{Augmented Lagrangian preconditioning}
\label{sec:AL}
The augmented Lagrangian (AL) approach replaces \eqref{eq:stokessys} with 
the equivalent linear system
\begin{equation}
\label{eq:augstokessys}
\begin{bmatrix}
	\mathbf A+\gamma \mathbf B^T\mathbf W^{-1} \mathbf B & \mathbf B^T\\
	\mathbf B & \mathbf 0
\end{bmatrix}
\begin{bmatrix}
	\mathbf u\\
	\mathbf p
\end{bmatrix}
=
\begin{bmatrix}
    \mathbf r_1+\gamma \mathbf B^T\mathbf W^{-1} \mathbf r_2\\
    \mathbf r_2
\end{bmatrix},
\end{equation}
for some positive definite $\mathbf W^{-1}\in \mathbb R^{m\times m}$.
Due to $\mathbf B\mathbf u = \mathbf r_2$, any solution to \eqref{eq:augstokessys} is also a solution to
\eqref{eq:stokessys}. In particular, if
$\mathbf r_2=0$, we obtain the more familiar form of the
incompressible Stokes
problem. 
We denote the augmented (1,1)-block by $\mathbf A_{\gamma} := \mathbf A+\gamma \mathbf B^T\mathbf
W^{-1}\mathbf B$ and consider the Schur complement
$\mathbf S_\gamma := \mathbf B\mathbf 
A_{\gamma}^{-1}\mathbf B^T$.
Using the Sherman-Morrison-Woodbury identity, one can derive that
\cite[Lemma 5.2]{Wechsung19}
\begin{equation}\label{eq:AL-Schur}
\mathbf S_\gamma^{-1}=\mathbf S^{-1} + \gamma \mathbf W^{-1}.
\end{equation}
Hence, an approximation of $\mathbf S_\gamma^{-1}$ can be obtained as
$\hat{\mathbf S}^{-1} + \gamma \mathbf W^{-1}$ with $\hat{\mathbf S}^{-1}
\approx \mathbf S^{-1}$.
We now aim to identify choices for $\hat{\mathbf S}$ and $\mathbf W$
that result in provably good approximations of $\mathbf S_\gamma^{-1}$
and in an effective and practical preconditioner. It turns out that
good choices for $\mathbf W$ are mass matrices, inverse
viscosity-weighted mass matrices and their lumped counterparts, i.e.,
the Schur complement approximations discussed in
\cref{sec:discreteschur}. 
Since we consider
candidates for $\hat{\mathbf S}$ and $\mathbf W$ that are spectrally equivalent
to the Schur complement $\mathbf S$ of the original system, we recall the
definition of spectral equivalence and introduce the corresponding constants.
The symmetric positive definite matrices $\hat{\mathbf S}$ and 
$\mathbf W$ are spectrally equivalent to the Schur
complement $\mathbf S$ if they satisfy 
\begin{equation}
	c_\mu \hat{\mathbf S} \leq \mathbf S \leq C_\mu \hat{\mathbf S}, 
	\quad d_\mu \mathbf W \leq \mathbf S \leq D_\mu \mathbf W,
	\quad e_\mu \mathbf W \leq \hat{\mathbf S} \leq E_\mu \mathbf W
	\label{eq:speceqvuiS1S2}
\end{equation}
with mesh independent constants  
$c_\mu, C_\mu, d_\mu, D_\mu, e_\mu, E_\mu>0$. 
Note that the third identity in \eqref{eq:speceqvuiS1S2} follows from the first
two, but possibly with suboptimal constants.
The subscript $\mu$ indicates that the constants may depend on the viscosity. 
For example, following
\cite[Lemma 3.1]{GrinevichOlshanskii09},
for $\hat{\mathbf S}$ being the inverse viscosity-weighted mass matrix
$\mathbf M_p(1/\mu)$, $c_\mu$ and $C_\mu$ can be chosen as
$c_0 \mu_{\max}^{-1}$ and 
$\mu_{\min}^{-1}$, respectively, with mesh independent constant $c_0>0$, 
$\mu_{\max}=\sup_{\Omega} \mu(\boldvar x)$ and 
$\mu_{\min}=\inf_{\Omega} \mu(\boldvar x)$.
The following lemma establishes a quantitative result for
the spectral equivalence on $\mathbf S_\gamma$ and $\hat{\mathbf
  S}_\gamma$.
\begin{lemma}[Eigenvalue bounds] Assume $\hat{\mathbf S}$ and $\mathbf
  W$ satisfy \eqref{eq:speceqvuiS1S2}.
  Then,
\label{lemma:spectrumest}
$\hat{\mathbf S}_\gamma = (\hat{\mathbf S}^{-1}+\gamma \mathbf W^{-1})^{-1}$
is spectrally equivalent to the Schur complement $\mathbf
S_\gamma$ of the augmented system, and
the spectrum satisfies
$\sigma\left(\hat{\mathbf S}_\gamma^{-1}\mathbf S_\gamma\right)\subset [f_\mu, F_\mu]$,
where
\begin{subequations}
\begin{align}
	f_\mu &:=
	\max\left(\frac{c_\mu}{1 + \gamma c_\mu E_\mu} + \frac{\gamma d_\mu}{1 + 
	\gamma d_\mu}~,~\frac{1 + \gamma e_\mu}{\max(1, c_\mu^{-1}) + \gamma
	e_\mu}\right), \label{eq:f_mu}\\
	F_\mu &:=
	\min\left(\frac{C_\mu}{1 + \gamma C_\mu e_\mu} + \frac{\gamma D_\mu}{1 +
	\gamma D_\mu}~,~\frac{1 + \gamma e_\mu}{\min(1, C_\mu^{-1}) + \gamma
	e_\mu}\right).\label{eq:F_mu}
\end{align}
\end{subequations}
Moreover, $f_\mu, F_\mu\rightarrow 1$ as $\gamma \rightarrow\infty$.
\end{lemma}
\begin{proof}
Consider the generalized eigenvalue problem $\mathbf S_\gamma \mathbf p= 
\lambda \hat{\mathbf S}_\gamma \mathbf p$  and 
let $\lambda_{\min}$ and $\lambda_{\max}$ be the smallest and largest eigenvalues,
respectively.
Observing that this generalized eigenvalue equation is equivalent to
$\hat{\mathbf S}_\gamma^{-1} \mathbf q= \lambda
\mathbf S_\gamma^{-1} \mathbf q$, where $\mathbf q = \mathbf S_\gamma \mathbf p$,
we find that $\lambda_{\min}$ and $\lambda_{\max}$ can be characterized by the generalized
Rayleigh quotients
\begin{equation}
	\lambda_{\min} = \min_{\mathbf q} \frac{\mathbf q^T
	\hat{\mathbf S}_\gamma^{-1}\mathbf q}{\mathbf q^{T}\mathbf S_\gamma^{-1}\mathbf q},
	\quad
	\lambda_{\max} = \max_{\mathbf q} \frac{\mathbf q^T
	\hat{\mathbf S}_\gamma^{-1}\mathbf q}{\mathbf q^{T}\mathbf S_\gamma^{-1}\mathbf q}.
\end{equation}
We now estimate $\lambda_{\min}$ and $\lambda_{\max}$ using these
Rayleigh quotients.
\begin{align*}
	\lambda_{\min} &= 
	\min_{\mathbf q} \dfrac{\mathbf q^T\left(\hat{\mathbf S}^{-1}+\gamma \mathbf W^{-1}
	\right)\mathbf q}{\mathbf q^T \left(\mathbf S^{-1}+\gamma \mathbf W^{-1}\right)
	\mathbf q}\\
	&\geq \min_{\mathbf q} 
	\dfrac{\mathbf q^T\hat{\mathbf S}^{-1}\mathbf q}
	{\mathbf q^T \left(\mathbf S^{-1}+\gamma \mathbf W^{-1}\right)\mathbf q} +
	\gamma\min_{\mathbf q} 
	\dfrac{\mathbf q^T\mathbf W^{-1}\mathbf q}
	{\mathbf q^T \left(\mathbf S^{-1}+\gamma \mathbf W^{-1}\right)\mathbf q}\\
                   &\geq \min_{\mathbf q}\dfrac{\mathbf q^T \hat{\mathbf S}^{-1}\mathbf q}{\mathbf q^T 
	\left(\frac{1}{c_\mu} \hat{\mathbf S}^{-1}+\gamma E_\mu \hat{\mathbf S}^{-1}\right)\mathbf q} 
    + \gamma \min_{\mathbf q} \dfrac{\mathbf q^T \mathbf W^{-1}\mathbf q}
	  {\mathbf q^T \left(\frac{1}{d_\mu}\mathbf W^{-1}+\gamma
	\mathbf W^{-1}\right)\mathbf q}\\
	&\geq \dfrac{1}{\frac{1}{c_\mu} + \gamma E_\mu} + \frac{\gamma d_\mu}{1 +
	\gamma d_\mu} = \frac{c_\mu}{1 + \gamma c_\mu E_\mu} + 
	\frac{\gamma d_\mu}{1+\gamma d_\mu}\label{eq:lambdamin},
\end{align*}
where \eqref{eq:speceqvuiS1S2} has been used in the first two inequalities.
Another estimation for $\lambda_{\min}$ is as follows
\begin{align*}
	\lambda_{\min} &=
	\min_{\mathbf q} \dfrac{1 + \gamma \dfrac{\mathbf q^T\mathbf W^{-1}\mathbf q}
	{\mathbf q^T\hat{\mathbf S}^{-1}\mathbf q}}{\dfrac{\mathbf q^T\mathbf S^{-1}\mathbf q}
	{\mathbf q^T\hat{\mathbf S}^{-1}\mathbf q} +\gamma
	\dfrac{\mathbf q^T\mathbf W^{-1}\mathbf q}
	{\mathbf q^T\hat{\mathbf S}^{-1}\mathbf q}}
        \geq \min_{\mathbf q} \dfrac{1 + \gamma\dfrac{\mathbf q^T\mathbf W^{-1}\mathbf q}
	{\mathbf q^T\hat{\mathbf S}^{-1}\mathbf q}}{\max(1,c_\mu^{-1}) +
    \gamma\dfrac{\mathbf q^T\mathbf W^{-1}\mathbf q}{\mathbf q^T\hat{\mathbf S}^{-1}\mathbf q}}\\
	&= \dfrac{1+\gamma \min_{\mathbf q}{\dfrac{\mathbf q^T\mathbf W^{-1}\mathbf q}
	{\mathbf q^T\hat{\mathbf S}^{-1}\mathbf q}}}{\max(1,c_\mu^{-1})+\gamma\min_{\mathbf q}
	{\dfrac{\mathbf q^T\mathbf W^{-1}\mathbf q}
	{\mathbf q^T\hat{\mathbf S}^{-1}\mathbf q}}}\quad
	(\text{since }x\mapsto\dfrac{1+\gamma x}{b+\gamma x} \text{ is increasing if } b\geq
	1)\\
	&\geq \dfrac{1 + \gamma e_\mu}{\max(1,c_\mu^{-1}) + \gamma e_\mu},
\end{align*}
where the first and the last inequality again use \eqref{eq:speceqvuiS1S2}.
Combining the above two estimates of $\lambda_{\min}$, we obtain that
$\lambda_{\min}\ge f_{\mu}$ with $f_{\mu}$ as defined in \eqref{eq:f_mu}.
Using similar arguments for $\lambda_{\max}$, one shows that
$\lambda_{\max}\le F_{\mu}$ with $F_{\mu}$ as defined in
\eqref{eq:F_mu}. Finally, it is easy to verify that $f_\mu, F_\mu\to
1$ as $\gamma\to\infty$, which ends the proof.
\end{proof}
\begin{remark}
For the case of $\hat{\mathbf S}=\mathbf W$, 
the eigenvalues of the generalized eigenvalue problem 
$\mathbf S_\gamma \mathbf x = \lambda \hat{\mathbf S}_\gamma\mathbf x$
are 
$$
	\lambda = \frac{1+\gamma}{\nu^{-1}+\gamma}, 
$$
where $\nu$ are the eigenvalues of the generalized eigenvalue problem
$\mathbf S\mathbf y = \nu \hat{\mathbf S} \mathbf y$, 
\cite[Section 2]{BenziOlshanskii06}. Our estimates reduce to the
same result assuming $c_\mu = \nu_{\min}$ and $C_\mu = \nu_{\max}$ since
$\hat{\mathbf S}=\mathbf W$ implies that
$e_\mu=E_\mu=1$, $c_\mu=d_\mu$, $C_\mu=D_\mu$, and hence
\begin{equation*}
	f_\mu = 
	\max\left(\frac{1+\gamma}{c_\mu^{-1}+\gamma}, 
	\frac{1+\gamma}{\max(1,c_\mu^{-1})+\gamma}\right) = 
	\frac{1+\gamma}{\nu_{\min}^{-1}+\gamma},\quad F_\mu = \frac{1+\gamma}{\nu_{\max}^{-1}+\gamma}.
\end{equation*}
\end{remark}
\begin{remark}
If one uses the block preconditioner $\mathbf P$ in \eqref{eq:precond} for the
augmented variable viscosity Stokes system \eqref{eq:augstokessys} and 
inverts $\mathbf A_\gamma$ exactly, i.e.,\ 
$\hat{\mathbf A}_\gamma^{-1} = \mathbf A_\gamma^{-1}$, a simple calculation
shows that
\begin{equation}
\mathbf P
\begin{bmatrix}
	\mathbf A_\gamma &\mathbf B^T\\
	\mathbf B&\mathbf 0
\end{bmatrix}=
\begin{bmatrix}
	\mathbf I_n & \mathbf A^{-1}\mathbf B^{T}(\mathbf I_m -
	\hat{\mathbf S}_\gamma^{-1}\mathbf S_\gamma)\\
	\mathbf{0} &  \hat{\mathbf S}_\gamma^{-1}\mathbf S_\gamma 
\end{bmatrix}.
\end{equation}
Hence, the condition number of the preconditioned system can be bounded in
terms of $f_\mu$ and $F_\mu$,
\begin{equation}
	\text{cond}\left(\mathbf P
\begin{bmatrix}
	\mathbf A_\gamma &\mathbf B^T\\
	\mathbf B&\mathbf 0
\end{bmatrix}
	\right)\leq \frac{\max(1, F_\mu)}{\min(1, f_\mu)}.	
\end{equation}
\end{remark}

Since the pressure mass matrix $\mathbf M_p$ and
the weighted pressure mass matrix $\mathbf M_p(1/\mu)$ are
spectrally equivalent to the Schur complement $\mathbf
S$, \cref{lemma:spectrumest} suggests two natural choices for
$(\hat{\mathbf{S}}, \mathbf W)$, namely  $(\mathbf
M_p(1/\mu),\mathbf
M_p)$ and $(\mathbf
M_p(1/\mu),\mathbf
M_p(1/\mu))$.
We call the resulting block preconditioners \eqref{eq:precond} AL
preconditioners $\mathbf P_1$ and $\mathbf P_2$:
\begin{equation}\label{eq:P1P2}
	\mathbf P_1: \hat{\mathbf S}_\gamma^{-1} = \mathbf M_p(1/\mu)^{-1}+\gamma 
	\mathbf M_p^{-1}\quad \text{and}\quad
	\mathbf P_2: \hat{\mathbf S}_\gamma^{-1} = \left(1+\gamma\right) \mathbf M_p(1/\mu)^{-1}.
\end{equation}

These two preconditioners are examined in \cref{table:exacttopleft} 
using the two-dimensional multi-sinker test problem detailed in \cref{sec:sinker}.
To exclusively study the Schur complement approximation, we use an
exact solve of $\mathbf A_{\gamma}$ in these experiments.
We find that the iteration counts decrease as
$\gamma$ increases for both preconditioners for all dynamic 
ratios, i.e., all viscosity contrasts. This numerically illustrates the results from
\cref{lemma:spectrumest}, i.e., that the Schur complement approximation
improves as $\gamma\to\infty$.

\begin{table}[ht]
	\label{table:exacttopleft}
	\centering
	\caption{Comparisons of AL preconditioners
          $\mathbf P_1$ and $\mathbf P_2$ defined in \eqref{eq:P1P2}
          for different $\gamma$ and different viscosity contrasts $\text{DR}(\mu)$.
           Shown are the number of FGMRES
           iterations to achieve $10^{6}$ residual reduction.
           The (1,1)-block of the system
          is solved exactly in each iteration using an LU
          factorization. The total number of velocity and pressure
          degrees of freedom is	394,754.}
	\begin{tabular}{lrrrrrrrrr}
		\toprule
		&\multicolumn{4}{c}{$\mathbf{P}_1$} && 
		\multicolumn{4}{c}{$\mathbf{P}_2$}\\[0.2cm]\
		$\gamma$  $\backslash \text{DR}(\mu)$ 
		& $10^4$ & $10^6$ & $10^8$ & $10^{10}$ 
		&& $10^4$ & $10^6$ & $10^8$ & $10^{10}$
		\\\midrule
		0    & 32 & 48 & 59 & 70   && 32 & 48 & 59 & 70\\
		10   & 7  &  9 & 10 & 13   && 10 & 16 & 20 & 24\\
		1000 & 2  &  3 &  4 &  5   &&  2 &  4 &  5 &  6\\
		\bottomrule
	\end{tabular}
\end{table}
 
\section{Robust multigrid for the (1,1)-block}
\label{sec:multigrid}
While adding the term $\mathbf B^T\mathbf W^{-1}\mathbf B$ makes it easier to approximate the Schur complement of 
the augmented system \eqref{eq:augstokessys}, 
inverting the resulting (1,1)-block becomes harder due to the 
large nullspace of the 
discrete divergence operator $\mathbf B$.
These difficulties can be seen in the numerical experiments in 
\cref{fig:topleftsolvefail}, where 
we study the convergence of
classical geometric and algebraic multigrid methods 
for inverting $\mathbf{A}_{\gamma}$, taken
from the two-dimensional multi-sinker test problem (see \cref{sec:sinker} for description)
with $\mathbf W=\mathbf M_p$.
We observe that standard geometric multigrid (GMG) schemes
with a Jacobi smoother fail to converge within $300$ iterations 
for $\gamma=10$.
Using algebraic multigrid (AMG) presents an improvement but the
number of iterations still increases significantly with $\gamma$. AMG
converges for moderate dynamic ratios of $\text{DR}(\mu)=10^4, 10^6$
for $\gamma=10$, but
fails to converge for larger dynamic ratios or $\gamma$. We tested
several AMG parameters and coarsening strategies but were not able to
improve these results. This is due to the
near-singularity of the operator, making it challenging to find
appropriate AMG parameters that lead to a good level hierarchy with low
operator complexity.

\begin{table}[ht]
	\label{fig:topleftsolvefail}
	\centering
	\caption{Number of FGMRES iterations preconditioned by an
          F-cycle of 
	geometric multigrid (GMG) and W-cycle of algebraic multigrid (AMG) for
	solving the augmented (1,1)-block  with
		$\mathbf W=\mathbf M_p$. The discretization is based on  $[\mathbb{Q}_3]^2\times
	\mathbb{P}_2^{\text{disc}}$ elements on a quadrilateral mesh
        with 296,450 unknowns.
        Shown is the number of FGMRES iterations to achieve $10^{6}$ residual reduction.  
	For GMG, we use 4 mesh levels with 5 Jacobi pre/post-smoothing
	steps on each level.
	For AMG, 5 W-cycles are applied per FGMRES iteration. The AMG hierarchy uses 8
	($\gamma=0$) and 9 ($\gamma=10$) mesh levels with 5 SSOR
        pre/post smoothing steps. Both methods use a direct solve on
        the coarse level.
	``-'' indicates failure of the solver to converge in 300 iterations.}
	\begin{tabular}{lccccccccc}
		\toprule
		&\multicolumn{4}{c}{Standard GMG} && 
		\multicolumn{4}{c}{BoomerAMG}\\[0.2cm]\
		$\gamma$  $\backslash \text{DR}(\mu)$ & $10^4$ & $10^6$ & $10^8$ & $10^{10}$ 
		&& $10^4$ & $10^6$ & $10^8$ & $10^{10}$ \\\midrule
		0    & 7 & 12 & 14 & 15 && 14 & 17  & 19 & 18 \\
		10   & - & - & - & -    && 34 & 123 & - & -\\
		\bottomrule
\end{tabular}
\end{table}

To address these difficulties, we use a multigrid scheme with
customized, $\gamma$-robust smoothing and transfer operators.
The design of the smoother and the prolongation operator is based
on a local characterization of the nullspace of the augmented term, i.e.,~the
space of discretely divergence-free functions. While a general framework
for robust multigrid was introduced by Sch\"oberl in~\cite{Schoeberl99},
establishing that the conditions for this framework are met is a
technical and highly element-specific task.

In~\cite{Schoeberl99}, robustness is proven for the
$[\mathbb{P}_2]^2\times \mathbb P_{0}$ element.
By adding bubble functions to the velocity space, this
result is extended to three dimensions for the $[\mathbb{P}_1\oplus
  B_3^F]^3\times 
\mathbb P_{0}$ element in~\cite{FarrellMitchellWechsung19}.
Higher-order discretizations (with non-constant pressure) were considered
in~\cite{FarrellMitchellScottWechsung20b}, where robustness is proven on specific
meshes for the Scott-Vogelius $[\mathbb{P}_k]^d\times
\mathbb{P}_{k-1}^\mathrm{disc}$ element.
While Scott-Vogelius elements enable
exact enforcement of the divergence constraint, the scheme
in~\cite{FarrellMitchellScottWechsung20b} requires 
barycentrically refined meshes at every level, and uses a block Jacobi smoother with
rather large block sizes. The latter amounts to a significant computational
effort, particularly in three dimensions.
Another class of discretizations that enforce the divergence constraint exactly are those building on $H(\mathrm{div})$ conforming elements.
It was shown in~\cite{arnold_1997, arnold_2000} that block Jacobi smoothers yield parameter robust multigrid methods in $H(\mathrm{div})$.
Using the same smoother and the local Discontinuous Galerkin formulation of~\cite{cockburn_2006},
in~\cite{HongKrausXuZikatanov16} a full multigrid convergence analysis is carried
out for nearly incompressible elasticity and the Stokes equations (with
constant viscosity). An advantage of working in these spaces is that no custom prolongation is necessary.

Unlike the existing work, here we consider quadrilateral and
hexahedral meshes.
Popular element choices on such meshes for the Stokes and Navier-Stokes equations are
the $[\mathbb{Q}_k]^d\times\mathbb{P}_{k-1}^{\text{disc}}$ and
$[\mathbb{Q}_k]^d\times\mathbb{Q}_{k-2}^{\text{disc}}$, $k\geq 2$, $d=2,3$ element
pairs. Here, we focus on the former case, but we remark that in numerical
experiments we also observed robust
performance of the same multigrid scheme for the latter element. 
We will construct smoothing and prolongation operators
and prove their robustness. This enables robust solution of high-order
discretized problem without similar mesh
limitations as required for Scott-Vogelius elements.

Before going into details of the smoother and the transfer operator construction
in \cref{fig:topleftsolve} we show
convergence results for the (1,1)-block obtained with the resulting multigrid scheme for the same
problem as in \cref{fig:topleftsolvefail}.
We can see that the multigrid scheme is able to
maintain similar convergence rates for $\gamma$ ranging from $0$ to 
$1000$ and for dynamic ratios $\text{DR}(\mu)$ up to $10^{10}$.

\begin{table}[ht]
	\label{fig:topleftsolve}
	\centering
	\caption{Number of FGMRES iterations preconditioned by F-cycle multigrid with
	the customized smoother and the customized prolongation operator for
	solving the (1,1)-block of the Stokes system with
	the $[\mathbb{Q}_3]^2\times\mathbb{P}_2^{\text{disc}}$ element on a quadrilateral
	mesh. 
	See \cref{fig:topleftsolvefail} for the description of the
        mesh and the solver setup.}
	\begin{tabular}{lrrrrrrrrr}
		\toprule
		&\multicolumn{4}{c}{$\mathbf W=\mathbf M_p$} && 
		\multicolumn{4}{c}{$\mathbf W=\mathbf M_p(1/\mu)$}\\[0.2cm]\
		$\gamma$  $\backslash \text{DR}(\mu)$ & $10^4$ & $10^6$ & $10^8$ & $10^{10}$ 
		&& $10^4$ & $10^6$ & $10^8$ & $10^{10}$ \\\midrule
		&\multicolumn{9}{c}{Robust smoother \& robust transfer}\\\midrule
		0    & 7 & 10 & 13 & 14 && 7 & 10 & 13 & 14\\
		10   & 6 & 12 & 14 & 14 && 6 &  9 & 11 & 11\\
		1000 & 7 & 14 & 17 & 17 && 7 & 12 & 14 & 14\\
		\bottomrule
\end{tabular}
\end{table}

For the analysis in the remainder of this section, we
restrict ourselves to 
$\mathbf{W}=\mathbf{M}_p$, i.e., the AL term $\mathbf B^T \mathbf
W^{-1}\mathbf B$ is the discrete form of 
$
\left(\Pi_{Q_h}\left(\divergence\vel\right),\Pi_{Q_h}\left(\divergence\veltest\right)\right),
$
where $\Pi_{Q_h}$ is the $L^2$-projection operator defined in \cref{sec:not}. 
We consider a shape regular mesh $\mathcal{T}_h$, defined in 
\cite{HeuvelineSchieweck07},  
with $\cup_{K\in \mathcal{T}_h} K = \bar{\Omega}$ 
in which
$(K_1)^o \cap (K_2)^o =
\emptyset$ for distinct elements $K_1,~K_2\in \mathcal{T}_h$. 
We denote by $h$ the mesh size of $\mathcal{T}_h$, defined as the largest 
diameter of any element $K\in \mathcal{T}_h$. 
To differentiate
the fine and coarse mesh operators $\mathbf A_{h,\gamma}$ and $\mathbf
A_{H,\gamma}$, respectively, we add subscripts $h$ or $H$.
We denote by $V_h^k$ and $Q_h^{k-1}$ the finite element spaces with
$[\mathbb{Q}_k]^d\times\mathbb{P}_{k-1}^{\text{disc}}$, $k\geq 2$
elements, i.e.,
\begin{align}
	&V_h^k := \left(\{\phi \in H^1_0(\Omega):\phi\vert_K \circ F_k\in \mathbb{Q}_k(\hat{K})\:\:\forall K\in \mathcal{T}_h\}\right)^d\\
    &Q_h^{k-1} := \{\phi \in L^2_0(\Omega):\phi\vert_K \circ F_k\in \mathbb{P}_{k-1}(\hat{K})\:\:\forall K\in \mathcal{T}_h\},
\end{align}
with $F_K:\hat{K}\rightarrow K$ being the mapping between the reference element 
$\hat{K}$ and $K$.

\subsection{Smoothing}
Many commonly used smoothers can be expressed as subspace correction methods. 
Here, we consider parallel subspace correction (PSC) methods, 
i.e., the residual correction on each subspace can be done in parallel.
Let $V_i$ be a decomposition of $V_h$, $V_h =\sum_i V_i$. One
PSC iteration smoothing step for a residual $\mathbf{r}^k_h$ is of the form 
\begin{equation*}
    \mathbf{u}^{k+1}_h = \mathbf{u}^k_h + \tau \mathbf{D}_{h,\gamma}^{-1}\mathbf{r}^k_h, \quad \text{where} \quad
	\mathbf D_{h,\gamma}^{-1} = \sum_i \mathbf I_i \mathbf A_i^{-1}\mathbf I_i^*
\end{equation*}
and $\mathbf I_i:V_h\rightarrow V_i$ is the natural inclusion,
$\mathbf A_i$ is the restriction of $\mathbf A_{h,\gamma}$ to 
subspace $V_i$ as
$(\mathbf A_i\boldvar u_i, \boldvar v_i):= (\mathbf A_{h,\gamma} \mathbf I_i
\boldvar u_i,\mathbf I_i\boldvar v_i)$,
and $\tau>0$ is a damping parameter. 

A key condition for a PSC smoother to be parameter-robust, 
i.e., the operator $\mathbf{D}_{h,\gamma}$
being spectrally equivalent to $\mathbf{A}_{h,\gamma}$ with constants independent
of $\gamma$,
is that the subspaces $V_i$ satisfy the \emph{kernel decomposition
property},~\cite[Theorem~4.1]{Schoeberl99}:
\begin{equation}\label{eq:kdp}
	\mathcal{N}_h = \sum_i (V_i \cap \mathcal{N}_h),
\end{equation}
where $\mathcal{N}_h:=\{\vel_h\in V_h: \Pi_{Q_h}\left(\divergence \vel_h\right)=0\}$ 
is the space of discretely divergence-free vector fields. 
Subspace decompositions $V_i$ satisfying this property have been
found on triangular and tetrahedral meshes for 
$[\mathbb{P}_2]^2\times \mathbb P_{0}$, 
$[\mathbb{P}_1\oplus B_3^F]^3\times \mathbb P_{0}$ 
and Scott-Vogelius discretizations.
In the latter case, the kernel is decomposed relying on the fact
that $\divergence V_h \subseteq Q_h$, which implies that discretely divergence-free
fields are also continuously divergence-free.  This is however not commonly
true for
other discretizations. 
For example,
for a $[\mathbb{P}_2]^2\times \mathbb P_{0}$ discretization, 
we can easily construct a field that has non-zero divergence but 
with divergence that integrates to zero on $K$ for all $K\in\mathcal{T}_h$, i.e., this
field is discretely but not continuously divergence-free.
The remedy for the $[\mathbb{P}_2]^2\times \mathbb P_{0}$ discretization
is to modify a discretely divergence-free field $\vel_0$ in the
\emph{interior} of each mesh element to obtain a continuously divergence-free
field $\vel_0+\boldvar\omega$, \cite{Schoeberl99}; 
and in addition that the modification $\boldvar\omega$ does
\emph{not} change the interpolated field $I_h(\vel_0)$, i.e.,
$I_h(\vel_0+\boldvar\omega)=I_h(\vel_0)$, 
where $I_h:(H^1_0(\Omega))^d\rightarrow V_h$ is a certain Fortin operator used in the
construction of the space decomposition $V_i$. 

For the higher order discretizations $[\mathbb Q_k]^d\times \mathbb
P_{k-1}^{\text{disc}}$, $k\geq 2$,
discretely divergence-free fields are not continuously divergence-free
in general either. We will use a modification similar to the one above
to make a discretely divergence-free field $\vel_0$ also continuously
divergence-free (see \cref{lemma:divfree}).
However, the modification $\boldvar\omega$ does not interpolate to $\boldvar 0$ with the
corresponding Fortin operator $I_h: (H^1_0(\Omega))^d\rightarrow V_h^k$ (as
in \cref{lemma:Ih}). Instead, we find that 
the modification is small: its $H^1$-norm is bounded above by the $H^1$-norm
of the original field $\vel_0$ up to a mesh-independent constant, i.e.,
\begin{equation*}
\|\boldvar\omega\|_1\preceq \|\vel_0\|_1.
\end{equation*}

The above observation motivates~\cref{prop:conDAequiv}, 
which provides a way to construct subspaces $V_i$ 
satisfying the kernel decomposition property~\eqref{eq:kdp}
for pairs $V_h\times Q_h$
for which $\divergence V_h \nsubseteq Q_h$.
In \cref{prop:conDAequiv}, we additionally
verify the stability of the
space decomposition, which implies
the $\gamma$-independent spectral equivalence of
$\mathbf{D}_{h,\gamma}$ following
\cite[Proposition 2.1]{FarrellMitchellScottWechsung20b}.
The proposition is presented in terms of a generic finite element
space pair $V_h\times Q_h$.
It holds under the assumptions summarized next.
\begin{assumption}\label{assump1}
  We make the following assumptions on the domain and the
  finite element discretization.  
\begin{enumerate}
	\item[(1)] $\Omega$ is a star-like domain with respect to some ball.
	\item[(2)] \{$\Omega_i$\} is an open covering of $\Omega$ such that
		for any mesh element $K\in \mathcal{T}_h$, $K^o \subset \Omega_i$
		if $\Omega_i \cap K\neq\emptyset$. 
	\item[(3)] $\{\rho_i\}$ 
		is a smooth partition of unity associated with \{$\Omega_i$\}
		satisfying
		$\|\rho_i\|_{L^{\infty}} \leq 1$, 
		$\|\rho_i\|_{W^{1, \infty}} \preceq h^{-1}$, 
		$\|\rho_i\|_{W^{2, \infty}} \preceq h^{-2}$ and  
		$\operatorname{supp}(\rho_i)\subset \Omega_i$.
	\item[(4)] $I_h: (H^1_0(\Omega)^d\rightarrow V_h$ is a Fortin operator, i.e., 
		$I_h$ is linear and continuous, 
		$I_h(\veltest_h) = \veltest_h$ for $\veltest_h\in V_h$, and 
		$(q_h, \divergence I_h(\veltest)) = (q_h, \divergence\veltest)$ for
		all $q_h\in Q_h$ and $\veltest \in V$.
	\item[(5)] For every $\vel_0\in \mathcal{N}_h$, there exist $\boldvar\omega \in
		V_{\text{loc}}$, 
		\begin{equation}
		\label{eq:vloc}
			V_{\text{loc}}:= \{\veltest\in \left(H_0^1(\Omega)\right)^d: \veltest=\boldvar 0 \text{
				on edges of element } K\:\:\forall K \in \mathcal{T}_h\},
		\end{equation}
		such that 
		$\divergence \left(\vel_0+\boldvar\omega\right)=0$,
		$\|\boldvar\omega\|_1\preceq\|\vel_0\|_1$ and 
		$\boldvar\omega_h=I_h(\boldvar\omega) \in V_{\text{loc}}$. 
\end{enumerate}
\end{assumption}
\begin{remark}
Note that 
(1), (3) and (4)
in \cref{assump1} are the same as in 
\cite[Proposition 2.2]{FarrellMitchellScottWechsung20b}.
However, assumptions (2) and (5) differ. In particular, the assumptions on
the open covering $\Omega_i$ are stricter and we assume the
existence of $\boldvar \omega$ in (5). These stricter
assumptions are needed to prove the existence of a splitting of
$\boldvar \omega_h$, which is then combined with the splitting of a
continuously divergence-free field to construct a splitting of a
discretely divergence-free field, as needed
to generalize the result from \cite{FarrellMitchellScottWechsung20b}
to settings where $\mathcal{N}_h\not=\mathcal{N}$. 
\end{remark}

\begin{proposition}
\label{prop:conDAequiv}
Under the conditions of \cref{assump1}, the space decomposition $\{V_i\}$ with 
	$		V_i:=\{I_h(\veltest):\veltest\in (H^1_0(\Omega))^d,\enskip\operatorname{supp}(\veltest)\subset
	\Omega_i\}$
satisfies
\begin{equation}
	\label{eq:con1}
	\inf_{\substack{\vel_h=\sum \vel_i\\ \vel_i\in V_i}}
	\sum_i \|\vel_i\|_1^2 \preceq h^{-2} \|\vel_h\|_0^2.
\end{equation}
Moreover, this decomposition satisfies the kernel decomposition
property
\eqref{eq:kdp}, and for any $\vel_0\in\mathcal{N}_h$ holds
\begin{equation}
	\label{eq:con2}
	\inf_{\substack{\vel_0=\sum \vel_{0,i}\\ \vel_{0,i}\in V_i\cap
        \mathcal N_h}}
	\sum_i \|\vel_{0,i}\|_1^2 \preceq h^{-4} \|\vel_{0}\|_0^2.
\end{equation}
\end{proposition}
\begin{proof}
	To prove \eqref{eq:con1}, for $\vel_h\in V_h$ we define 
$\vel_i := I_h(\rho_i \vel_h)\in V_i$, which implies that
$\sum_i \vel_i = I_h\left(\sum_i \rho_i \vel_h\right) = I_h(\vel_h) = \vel_h$
and 
\begin{equation}
  \begin{aligned}
    \|\vel_i\|_{H^1(\Omega_i)}^2&\preceq\|\rho_i\vel_h\|_{H^1(\Omega_i)}^2\\
    &\leq \|\vel_h\|_{L^2(\Omega_i)}^2\|\nabla\rho_i\|_{L^{\infty}(\Omega_i)}^2 + 
    \|\vel_h\|_{H^1(\Omega_i)}^2\|\rho_i\|_{L^{\infty}(\Omega_i)}^2\\
    &\preceq h^{-2}\|\vel_h\|_{L^2(\Omega_i)}^2.
  \end{aligned}
\end{equation}
To show \eqref{eq:con2},  
let $\vel_0\in\mathcal{N}_h$ be some discretely divergence free velocity field. Then,
using (5) in \cref{assump1}, there exists 
$\boldvar\omega\in V_{\text{loc}}$ such that 
$\divergence\left(\vel_0+\boldvar\omega\right) = 0$.
Therefore,  
by \cite[Theorem 3.3]{Rannacher00} in 2D and 
by \cite{CostabelMcIntosh10} in 3D
there exists $\boldvar \Phi\in H^2(\Omega)$ such that 
$\nabla\times\boldvar\Phi = \vel_0+\boldvar\omega$,
$\|\boldvar\Phi\|_2 \preceq \|\vel_0+\boldvar\omega\|_1$ 
and $\|\boldvar\Phi\|_1 \preceq \|\vel_0+\boldvar\omega\|_0$.
Based on this, we use the identity $\vel_0 =
(\vel_0+I_h(\boldvar\omega))-I_h(\boldvar\omega)$ and construct
a splitting and estimates for $\vel_0+I_h(\boldvar\omega)$ and $I_h(\boldvar\omega)$
separately.

For $\vel_0+I_h(\boldvar\omega)$, 
the splitting and the estimates are obtained similarly to the 
arguments for \eqref{eq:con1}: we define
$\boldvar v_{i,0} := I_h\left(\nabla\times\left(\rho_i\boldvar \Phi\right)\right)
\in \mathcal{N}_h\cap V_i$
and observe that
	\begin{equation*}
		\begin{aligned}
		\sum_i \boldvar v_{i,0} 
			&= \sum_i I_h\left(\nabla\times\left(\rho_i\boldvar \Phi\right)\right)
			=I_h\Big(\nabla \times \Big(\sum_i \rho_i\boldvar\Phi\Big)\Big)\\
			&=I_h\left(\nabla\times\boldvar\Phi\right)=I_h\left(\vel_0+\boldvar\omega\right)
			=\vel_0 + \boldvar\omega_h
		\end{aligned}
	\end{equation*}
    and (by the same arguments as in \cite[Proposition~2.2]{FarrellMitchellScottWechsung20b})
	\begin{equation*}
			       \|\boldvar v_{i,0}\|_{H^1(\Omega)}^2
			\preceq h^{-4} \|\boldvar \Phi\|_{L^2(\Omega_i)}^2 +
			      h^{-2} \|\boldvar \Phi\|_{H^1(\Omega_i)}^2 + 
				  \|\boldvar \Phi\|_{H^2(\Omega_i)}^2.
	\end{equation*}
Summing over $i$ and denoting the maximum number of subspace overlaps
by $N_o$, we obtain
	\begin{equation}
	\label{eq:vvi}
		\begin{aligned}
			\sum_i\|\boldvar v_{i,0}\|_{H^1(\Omega)}^2
		&\preceq N_o\left(
		h^{-4} \|\boldvar \Phi\|_{L^2(\Omega)}^2 +
		h^{-2} \|\boldvar \Phi\|_{H^1(\Omega)}^2 + 
		\|\boldvar \Phi\|_{H^2(\Omega)}^2 \right)\\
		&\preceq N_o\left(
		h^{-4} \|\vel_0+\boldvar \omega\|_{L^2(\Omega)}^2 +
		h^{-2} \|\vel_0+\boldvar \omega\|_{L^2(\Omega)}^2 + 
		       \|\vel_0+\boldvar \omega\|_{H^1(\Omega)}^2 \right)\\
		&\preceq N_o h^{-4} \|\vel_0+\boldvar \omega\|_{L^2(\Omega)}^2\preceq N_o h^{-4} \|\vel_0\|_{H^1(\Omega)}^2,
		\end{aligned}
	\end{equation}
where the last inequality uses $\|\boldvar\omega\|_1\preceq \|\vel_0\|_1$.

For $I_h(\boldvar\omega) = \boldvar\omega_h$, we first assign each mesh
element $K\in\mathcal{T}_h$ an index $i_K$ such that 
$K\cap\Omega_{i_K}\neq\emptyset$
and define the set $I_i$ as the union of elements with index $i$, i.e.,
	$I_i = \bigcup_{K\in\mathcal{T}_h, i_K=i} K$.
Then, given $\boldvar \omega_h\in V_{\text{loc}}$ we define
$	\boldvar\omega_i := \chi_{I_i} \boldvar \omega_h$, 
where $\chi_{I_i}$ be the indicator function of the set $I_i$. 
By definition, $\bar{\Omega} = \bigcup_i I_i$ with pairwise disjoint
$I_i$, and hence 
$\sum_i \boldvar \omega_i = \boldvar\omega_h$. 
From (2) in \cref{assump1} and $\boldvar\omega_h\in V_{\text{loc}}$, 
$\operatorname{supp}(\boldvar\omega_i) = \bigcup_{K\in\mathcal{T}_h, i_K=i} K^o 
\subset \Omega_i$. 
Observing that 
\begin{align*}
	\Pi_{Q_h} \left(\divergence\boldvar\omega_h\right) &= 
	\Pi_{Q_h} \left(\divergence\left(\vel_0 + \boldvar\omega_h\right)\right) = 
	\Pi_{Q_h} \left(\divergence I_h\left(\vel_0 + \boldvar\omega \right)\right)\\
	&= \Pi_{Q_h} \left(\divergence \left(\vel_0 + \boldvar\omega \right)\right) = 0,
\end{align*}
we have $\boldvar\omega_h \in \mathcal{N}_h$ and since
$\operatorname{supp}(\boldvar \omega_i)=I_i$ are disjoint, 
\begin{align*}
0
=\int_{\Omega}\left(\Pi_{Q_h}\divergence \boldvar\omega_h\right)^2d\boldvar x
= \int_{\Omega}\sum_i\left(\Pi_{Q_h}\divergence \boldvar\omega_i\right)^2d\boldvar x
= \sum_i\int_{I_i}\left(\Pi_{Q_h}\divergence \boldvar\omega_i\right)^2d\boldvar x.
\end{align*}
Therefore, $\boldvar \omega_i\in V_i\cap \mathcal{N}_h$.
Using $\|\boldvar\omega\|_1\preceq\|\vel_0\|_1$, we obtain the estimate
\begin{equation}
	\label{eq:wi}
	\sum_i \|\boldvar\omega_i\|_{H^1(\Omega)}^2 = 
	\sum_{i} \|\boldvar\omega_h\|_{H^1(I_i)}^2 =
	\|\boldvar\omega_h\|_{H^1(\Omega)}^2 
	\preceq\|\vel_0\|_{H^1(\Omega)}^2.
\end{equation}
We now combine the splitting $\veltest_i$ for
$\vel_0+\boldvar\omega_h$, and $\boldvar\omega_i$ for
$\boldvar\omega_h$ defining $\vel_{0,i} := \veltest_i -
\boldvar\omega_i$. Clearly, 
$\sum_i\vel_{0,i} = \sum_i\veltest_i - \sum_i\boldvar\omega_i = \vel_0 +
\boldvar\omega_h - \boldvar\omega_h = \vel_0$ 
and from \eqref{eq:vvi}, \eqref{eq:wi}, we conclude that
\begin{equation*}
\sum_i\|\boldvar \vel_{i,0}\|_{H^1(\Omega)}^2
\preceq \left(1+N_o h^{-4}\right) \|\vel_0\|_{H^1(\Omega)}^2
\preceq h^{-4} \|\vel_0\|_{H^1(\Omega)}^2, 
\end{equation*}
which shows \eqref{eq:con2} and ends the proof.
\end{proof}

\subsubsection*{Application to $[\mathbb Q_k]^d\times \mathbb P_{k-1}^{\text{disc}}$, $k\geq 2$,
  elements}\label{sec:qkpk-1}
We now use \cref{prop:conDAequiv} 
to show that the PSC smoother with the space decomposition 
\begin{equation}
\label{eq:viim}
	V_h = \sum_{i} V_i, \quad 
	V_i = \{\veltest\in V_h: \operatorname{supp}(\veltest) \subset \operatorname{star}(v_i)\},
\end{equation}
where 
$
	\operatorname{star}(v_i):= \bigcup_{K\in \mathcal{T}_h:~v_i\in K} K
$ 
for $v_i$ being a vertex of $\mathcal{T}_h$ (see \cref{fig:starvi}) is
para\-me\-ter-robust for the discretization  $[\mathbb Q_k]^d\times \mathbb
P_{k-1}^{\text{disc}}$, $k\geq 2$. 
The proof can be summarized into three steps:
(1) construct a Fortin operator $I_h$ mapping functions in $(H^1_0(\Omega))^d$ 
to $V_h^k$;
(2) for any discretely divergence-free field $\vel_0$, 
prove the existence of $\boldvar \omega\in V_{\text{loc}}$ such that
$\vel_0+\boldvar\omega$ is continuously divergence-free, and
(3) apply \cref{prop:conDAequiv} to conclude the
$\gamma$-independent spectral equivalence of $\mathbf{D}_{h,\gamma}$
and $\mathbf{A}_{h,\gamma}$. We show steps (1) and (2) in \cref{lemma:Ih}
and \cref{lemma:divfree}, respectively.

\begin{figure}
\centering
\begin{tikzpicture}[scale=0.5]
\draw[step=2cm,thin,black] (0,0) grid (8,8);
\fill[blue, opacity=0.2] (4-1.8,4-1.8) rectangle (4+1.8,4+1.8);
\draw[->] (4+0.1,4+0.1) [out=10, in=230] to (4+0.9, 4+0.5) node[right,xshift=-0.11cm,yshift=0.11cm]{$v_i$};
\foreach \x in {0,2,4,6}
  \foreach \y in {0,2,4,6}
   {
    \draw [fill=black](\x,\y) circle (0.08cm);
	\draw [fill=black](\x+1.57,\y) circle (0.08cm);
	\draw [fill=black](\x+1-.57,\y) circle (0.08cm);
	\draw [fill=black](\x+2,\y) circle (0.08cm);
    \draw [fill=black](\x,\y+1.57) circle (0.08cm);
	\draw [fill=black](\x+1.57,\y+1.57) circle (0.08cm);
	\draw [fill=black](\x+1-.57,\y+1.57) circle (0.08cm);
	\draw [fill=black](\x+2,\y+1.57) circle (0.08cm);
    \draw [fill=black](\x,\y+1-.57) circle (0.08cm);
	\draw [fill=black](\x+1.57,\y+1-.57) circle (0.08cm);
	\draw [fill=black](\x+1-.57,\y+1-.57) circle (0.08cm);
    \draw [fill=black](\x+2,\y+1-.57) circle (0.08cm);
    \draw [fill=black](\x,\y+2) circle (0.08cm);
	\draw [fill=black](\x+1.57,\y+2) circle (0.08cm);
	\draw [fill=black](\x+1-.57,\y+2) circle (0.08cm);
    \draw [fill=black](\x+2,\y+2) circle (0.08cm);
   }
\foreach \x in {0,2,4,6}
  \foreach \y in {0,2,4,6}
   {
	   \draw [red,very thick,line cap=round] (\x,\y      -0.0) -- (\x,\y      +0.12);
	   \draw [red,very thick,line cap=round] (\x,\y+1-.57-0.12) -- (\x,\y+1-.57+0.12);
	   \draw [red,very thick,line cap=round] (\x,\y+1+.57-0.12) -- (\x,\y+1+.57+0.12);
	   \draw [red,very thick,line cap=round] (\x,\y+2    -0.12) -- (\x,\y+2    +0.0);
	   \draw [red,very thick,line cap=round] (\x+2,\y      -0.0) -- (\x+2,\y      +0.12);
	   \draw [red,very thick,line cap=round] (\x+2,\y+1-.57-0.12) -- (\x+2,\y+1-.57+0.12);
	   \draw [red,very thick,line cap=round] (\x+2,\y+1+.57-0.12) -- (\x+2,\y+1+.57+0.12);
	   \draw [red,very thick,line cap=round] (\x+2,\y+2    -0.12) -- (\x+2,\y+2    +0.0);
	   \draw [red](\x+1.57 ,\y+1-.57) circle (0.1cm);
	   \draw [red](\x+1-.57,\y+1-.57) circle (0.1cm);
	   \draw [red](\x+1.57 ,\y+1+.57) circle (0.1cm);
	   \draw [red](\x+1-.57,\y+1+.57) circle (0.1cm);
	   \draw [red,very thick,line cap=round] (\x      -0.0,\y) -- (\x      +0.12,\y);
	   \draw [red,very thick,line cap=round] (\x+1-.57-0.12,\y) -- (\x+1-.57+0.12,\y);
	   \draw [red,very thick,line cap=round] (\x+1+.57-0.12,\y) -- (\x+1+.57+0.12,\y);
	   \draw [red,very thick,line cap=round] (\x+2    -0.12,\y) -- (\x+2    +0.0,\y);
	   \draw [red,very thick,line cap=round] (\x      -0.0,\y+2) -- (\x      +0.12,\y+2);
	   \draw [red,very thick,line cap=round] (\x+1-.57-0.12,\y+2) -- (\x+1-.57+0.12,\y+2);
	   \draw [red,very thick,line cap=round] (\x+1+.57-0.12,\y+2) -- (\x+1+.57+0.12,\y+2);
	   \draw [red,very thick,line cap=round] (\x+2    -0.12,\y+2) -- (\x+2    +0.0,\y+2);
   }
\end{tikzpicture}
\hspace{1.0cm}
\begin{tikzpicture}[scale=0.5]
\draw[step=2cm, black, ultra thin] (0,0) grid (8,8);
\draw[step=4cm, black, thick] (0,0) grid (8,8);
\fill[blue,opacity=0.2] (2-1.8,2-1.8) rectangle (2+1.8,2+1.8);
\fill[blue,opacity=0.2] (2-1.8,6-1.8) rectangle (2+1.8,6+1.8);
\fill[blue,opacity=0.2] (6-1.8,2-1.8) rectangle (6+1.8,2+1.8);
\fill[blue,opacity=0.2] (6-1.8,6-1.8) rectangle (6+1.8,6+1.8);
\foreach \x in {0,2,4,6}
  \foreach \y in {0,2,4,6}
   {
    \draw [fill=black](\x,\y) circle (0.08cm);
    \draw [fill=black](\x+1.57,\y) circle (0.08cm);
    \draw [fill=black](\x+1-.57,\y) circle (0.08cm);
    \draw [fill=black](\x+2,\y) circle (0.08cm);
    \draw [fill=black](\x,\y+1.57) circle (0.08cm);
    \draw [fill=black](\x+1.57,\y+1.57) circle (0.08cm);
    \draw [fill=black](\x+1-.57,\y+1.57) circle (0.08cm);
    \draw [fill=black](\x+2,\y+1.57) circle (0.08cm);
    \draw [fill=black](\x,\y+1-.57) circle (0.08cm);
    \draw [fill=black](\x+1.57,\y+1-.57) circle (0.08cm);
    \draw [fill=black](\x+1-.57,\y+1-.57) circle (0.08cm);
    \draw [fill=black](\x+2,\y+1-.57) circle (0.08cm);
    \draw [fill=black](\x,\y+2) circle (0.08cm);
    \draw [fill=black](\x+1.57,\y+2) circle (0.08cm);
    \draw [fill=black](\x+1-.57,\y+2) circle (0.08cm);
    \draw [fill=black](\x+2,\y+2) circle (0.08cm);
   }
\end{tikzpicture}
\caption{On the left, the blue region is the interior of $\operatorname{star}(v_i)$
and the red line regions are the integration domains of the Scott-Zhang interpolant $I_2^3$
for the $[\mathbb{Q}_3]^2$ element.
On the right, the blue regions are domains of local problems solved
in the robust prolongation operator.}
\label{fig:starvi}
\end{figure}
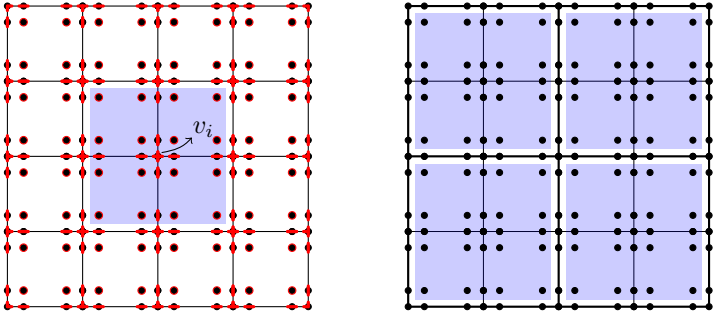

\begin{lemma}
\label{lemma:Ih}
For every vertex $v_i$, define 
$\Omega_i$ to be the interior of $\operatorname{star}(v_i)$. 
There exists an interpolation operator $I_h:(H^1_0(\Omega))^d\rightarrow V_h^k$ such that 
	\begin{enumerate}
		\item[(a)] $I_h$ is linear and continuous, 
		\item[(b)] $(q_h, \divergence I_h(\veltest)) = (q_h, \divergence \veltest)$
			for all $q_h\in Q_h^{k-1}$ and $\veltest \in (H^1_0(\Omega))^d$,
		\item[(c)] $I_h(\veltest_h) = \veltest_h$ for $\veltest_h\in V_h^k$,
		\item[(d)] $\operatorname{supp}\left(I_h(\veltest)\right)\subset \operatorname{star}(v_i)
			  \quad 
			  \forall\veltest\in (H^1_0(\Omega))^d$ such
			  that $\operatorname{supp}(\veltest)\subset\Omega_i$.
	\end{enumerate}
\end{lemma}
\begin{proof}
	Our goal is to construct $\tilde{I}_h: (H^1_0(\Omega))^d\rightarrow V_h^k$ that
satisfies the assumptions (A1) in 
\cite[Lemma 2.5]{FarrellMitchellScottWechsung20b}, which coincide with
(a)--(d), except that in (b), $Q_h^{k-1}$ is replaced by $Q_h^0$. Once we have
verified these conditions for $\tilde{I}_h$, the result in
\cite{FarrellMitchellScottWechsung20b} together with
the local inf-sup stability of $V_h^k\times Q_h^{k-1}$ on mesh element
$K\in \mathcal{T}_h$ guarantees 
	the existence of a linear map $I_h:(H^1_0(\Omega))^d \rightarrow V_h^k$
satisfying (a)--(d). Thus, what remains is to construct an appropriate
operator ${\tilde I}_h$.

In \cite{HeuvelineSchieweck07}, the macro element technique is used 
for proving the inf-sup stability of the pair 
$V_h^k\times Q_h^{k-1}$. This
involves a proof of a local inf-sup condition on macro elements, with macro
elements being the mesh elements $K\in \mathcal{T}_h$, and a
global inf-sup condition proof for the pair  $(V_h^2,Q_h^0)$. 
In the global inf-sup stability proof, a continuous divergence-preserving
interpolation $I_1:(H_0^1(\Omega))^d\rightarrow V_h^2$ 
is constructed, which
satisfies 
\begin{align}
&(\divergence I_1(\veltest), \presstest_h) = (\divergence \veltest, \presstest_h)
	\quad \forall \presstest_h\in Q_h^0, \nonumber\\
&\operatorname{supp}(I_1(\veltest)) \subset \operatorname{star}(v_i)\quad 
	\forall \veltest\in (H^1_0(\Omega))^d \text{ such that }\operatorname{supp}(\veltest) \subset \Omega_i\label{eq:I1con2}.
\end{align}
Let $I^k_2: (H_0^1(\Omega))^d\rightarrow V_h^k$ be the Scott-Zhang interpolation
\cite{ScottZhang90} operator with integration domains shown in red
in
\cref{fig:starvi}. $I_2^k$ satisfies
\begin{align}
&I_2^k(\veltest_h) = \veltest_h\quad\forall \veltest_h\in
  V_h^k,\nonumber \\
&\operatorname{supp}(I_2^k(\veltest)) \subset \operatorname{star}(v_i)\quad
	\forall \veltest\in (H^1_0(\Omega))^d \text{ such that }\operatorname{supp}(\veltest) \subset \Omega_i\label{eq:I2con2}.
\end{align}
Define $\tilde{I}_h(\veltest) := I_2^k(\veltest)+I_1(\veltest-I_2^k(\veltest))$.
We now show that $\tilde{I}_h(\veltest)$ satisfies (A1).
First, since both $I_1$ and $I_2^k$ are linear and continuous, 
$\tilde{I}_h$ is also linear and continuous. 
Second, 
\begin{align*}
	(\divergence\tilde{I}_h(\veltest), \presstest) 
	&= (\divergence I_2^k (\veltest), \presstest)+
	  (\divergence I_1(\veltest-I_2^k (\veltest)), \presstest)\\
	&=(\divergence I_2^k (\veltest), \presstest) 
	+ (\divergence (\veltest-I_2^k (\veltest)), \presstest)
    =(\divergence\veltest, \presstest).
\end{align*}
Third, 
$
    \tilde{I}_h(\veltest_h) = I^k_2(\veltest_h) + I_1(\veltest_h-I^k_2(\veltest_h))
    =\veltest_h + I_1(\boldvar 0)
    =\veltest_h$.
Last, from \eqref{eq:I1con2}, \eqref{eq:I2con2}, we have
	$\tilde{I}_h(\veltest)\in V_i$ for all $\veltest\in (H^1_0(\Omega))^d$ such that
	$\operatorname{supp}(\veltest)\subset \Omega_i$.  
\end{proof}

Finally, it remains to show that any discretely
divergence-free field $\vel_0$ can be modified in the interior of each
mesh element to obtain a continuously divergence-free field.
\begin{lemma}\label{lemma:divfree}
For each $\vel_0\in \mathcal{N}_h^k:=\{\vel_h\in V_h^{k}:
\Pi_{Q_h^{k-1}} (\nabla \cdot\vel_h) = 0\}$, there exists
	$\boldvar\omega\in V_{\text{loc}}$, defined in \eqref{eq:vloc},
such that 
\begin{equation}\label{eq:lemma45}
\divergence(\vel_0+\boldvar \omega)=0,\quad 
\|\boldvar \omega\|_1^2 \preceq \|\vel_0\|_1^2,\quad
\boldvar w_h = I_h(\boldvar \omega)\in V_{\text{loc}},
\end{equation}
	for $I_h:(H^1_0(\Omega))^d\rightarrow V_h^k$ from \cref{lemma:Ih}.
\end{lemma}
\begin{proof} 
Let
$
	Q_{\text{loc}} :=\{q\in L^2_0(\Omega): \Pi_{Q_h^0}(q)=0\}.
$
The pair $V_{\text{loc}}\times Q_{\text{loc}}$ is inf-sup stable for
the bilinear form 
\begin{equation*}
B((\vel, \press), (\veltest, \presstest)) :=  a(\vel, \veltest)-(\divergence
       \veltest,\press)-(\divergence \vel, \presstest).
\end{equation*}
Let $(\boldvar w,r)$ be the solution of the variational problem
\begin{equation}
	B((\boldvar\omega,r), (\veltest,\presstest)) = -(\divergence\vel_0,\presstest)
	\quad \forall(\veltest,\presstest) \in V_{\text{loc}}\times Q_{\text{loc}}.
\end{equation}
Choosing $\veltest=\boldvar 0$, we get
	\begin{equation}
		\label{eq:df}
		(\divergence(\boldvar\omega+\vel_0),q) = 0\quad \forall q\in Q_{\text{loc}}. 
	\end{equation}
From the divergence theorem, we have
$\Pi_{Q_h^0}\left(\nabla\cdot \boldvar\omega\right) = 0$, and
since $Q_h^0 \subset Q_h^{k-1}$, 
$\Pi_{Q_h^0}\left(\nabla\cdot\vel_0\right) = 0$. 
We therefore have $\Pi_{Q_h^0}\left(\nabla\cdot
(\boldvar\omega+\vel_0)\right) = 0$ and so 
$\divergence\left(\boldvar\omega+\vel_0\right)\in Q_{\text{loc}}$. 
From \eqref{eq:df},  
we get $\divergence (\boldvar\omega+\vel_0) = 0$. 
By the inf-sup stability, 
we have
\begin{equation}
       \|\boldvar\omega\|_1
	   \preceq
	   \sup_{\substack{\veltest\in V_{\text{loc}}\\\presstest_h\in Q_{\text{loc}}}}
       \frac{B((\boldvar \omega, \press), (\veltest, \presstest))}
                {\sqrt{\|\veltest\|_1^2+\|\presstest\|_0^2}}
       \leq\sup_{\substack{\veltest\in V_{\text{loc}}\\\presstest_h\in Q_{\text{loc}}}}
       \frac{\|\divergence\vel_0\|_0\|\presstest\|_0}
                {\sqrt{\|\veltest\|_1^2+\|\presstest\|_0^2}}\\
       \leq \|\divergence \vel_0\|_0.
\end{equation}
Now, the middle statement in \eqref{eq:lemma45} follows from $ \|\divergence\vel_0\|_{0}\preceq\|\vel_0\|_{1}$.
Lastly, from the locality of $I_h$, $\boldvar\omega_h$ remains
$\boldvar 0$ on
edges of elements and hence $\boldvar\omega_h \in V_{\text{loc}}$.
\end{proof}

The interpolation operator $I_h$ obtained in \cref{lemma:Ih}
and $\boldvar \omega\in V_{\text{loc}}$ from
\cref{lemma:divfree} satisfy (4) and (5) of
\cref{assump1}. 
By applying \cref{prop:conDAequiv}, we obtain the
estimates
\eqref{eq:con1}, \eqref{eq:con2} with 
\begin{equation}\label{eq:Vi-alternative}
\begin{aligned}
	&\Omega_i:= \text{ the interior of $\operatorname{star}(v_i)$},\\
	\text{and }& V_i:=\{I_h(\veltest):\veltest\in 
	(H^1_0(\Omega))^d,\enskip\operatorname{supp}(\veltest)\subset \Omega_i\}.
\end{aligned}
\end{equation}
In addition with the inf-sup stability of 
$V_h^k\times Q_h^{k-1}$ for the mixed problem
\begin{equation*}
B((\vel, \press), (\veltest, \presstest)) :=  a(\vel, \veltest)-(\divergence
       \veltest,\press)-(\divergence \vel, \presstest),
\end{equation*}
we apply \cite[Proposition 2.1]{FarrellMitchellScottWechsung20b} and 
conclude the $\gamma$-independent
spectral equivalence of $\mathbf D_{h,\gamma}$ and  $\mathbf A_{h,\gamma}$.
\begin{remark}
  Note that the definition of the subspace $V_i$ in
  \eqref{eq:Vi-alternative} is equivalent to the definition
  \eqref{eq:viim}. This is since the interpolation
  $I_h$ satisfies
  $\operatorname{supp}\left(I_h(\veltest)\right)\subset
	\operatorname{star}(v_i)$ for $\veltest\in (H^1_0(\Omega))^d$ with
  $\operatorname{supp}(\veltest)\subset\Omega_i$. Our numerical
  implementation is based on the definition \eqref{eq:viim}.
\end{remark}

\subsection{Prolongation}
A parameter-robust multigrid solver relies on a prolongation operator,
$\tilde{P}_H$, that is continuous in the energy norm with $\gamma$-independent
constants \cite{Schoeberl99}, i.e., 
\begin{equation}
	\label{eq:conP}
	\|\tilde{P}_H(\vel_H)\|_{\mathbf{A}_{h,\gamma}}\preceq \|\vel_H\|_{\mathbf{A}_{H,\gamma}}.
\end{equation}
This can be obtained by modifying the standard prolongation
$P_H$ so that it maps divergence-free fields on the coarse grid to
nearly divergence-free fields on the fine grid. 
A similar modification as done for
$[\mathbb{P}_2]^2\times \mathbb P_{0}$\cite{BenziOlshanskii06,Schoeberl99},
$[\mathbb{P}_2\oplus B_3^F]^3\times \mathbb P_{0}$\cite{FarrellMitchellWechsung19}
and Scott-Vogelius discretizations\cite{FarrellMitchellScottWechsung20b}
applies for the $[\mathbb Q_k]^d\times \mathbb P_{k-1}^{\text{disc}}$, $k\geq 2$, $d=2,3$ discretization
that we consider.
We define $\tilde{P}_H:V_H^k\rightarrow
V_h^k$ as
\begin{equation}
	\label{eq:Ptilde}
	\tilde{P}_H (\boldvar u_H) := P_H(\boldvar u_H) - \tilde{\boldvar u}_h,
\end{equation}
where $\tilde{\boldvar u}_h$ is the solution of 
\begin{equation*}
	a_{h,\gamma}(\tilde{\boldvar u}_h, \hat{\veltest}_h) = 
	\gamma (\Pi_{Q_h^{k-1}}(P_H(\boldvar u_H)), \Pi_{Q_h^{k-1}}(\nabla\cdot(\hat{\veltest}_h)))\quad \text{for all } \hat{\veltest}_h \in \hat{V}_h,
\end{equation*}
with 
	$\hat{V}_h := \{\boldvar v_h \in V_h^k,~\operatorname{supp}(\veltest_h)\subset K \text{ for
some } K \in\mathcal{T}_H\}$.
In the next lemma we will show that the conditions of~\cite[Proposition 3.1]{FarrellMitchellScottWechsung20b} are satisfied, and that hence
$\tilde P_H$ is continuous in the sense of~\eqref{eq:conP}.
We first
define the coarse and fine pressure spaces as
\begin{align*}
  \tilde{Q}_H&:=\{q\in L^2: q \text{ is constant on each coarse element } K\in \mathcal{T}_H\},\\
  \hat{Q}_h &:= \{q_h \in Q_h^{k-1}: \Pi_{\tilde{Q}_H} q_h = 0\},
\end{align*}
and then summarize the properties that imply \eqref{eq:conP} next.
\begin{lemma}\label{lemma:P}
The following statements are satisfied:
\begin{itemize}
\item[(a)] $Q_h^{k-1} = \tilde{Q}_H\oplus\hat{Q}_h$,
\item[(b)] $(\divergence\hat{\veltest}_h, \tilde{q}_H)=0$ for all 
  $\tilde{q}_H\in\tilde{Q}_H,\hat{\veltest}_h\in\hat{V}_h$
\item[(c)] the pairing $\hat{V}_h\times\hat{Q}_h$ is inf-sup stable 
  for the discretization of \eqref{eq:weakstokes}, i.e.,
  \begin{equation}
    \label{eq:hatsp}
    \inf_{\hat{q}_h\in\hat{Q}_h}\sup_{\hat{\veltest}_h\in\hat{V}_h}
    \frac{(\hat{q}_h, \divergence{\hat{\veltest}_h})}
	 {\|\hat{\veltest}_h\|_1\|\hat{\presstest}\|_0}\geq c>0.
  \end{equation}
\item[(d)] $P_H: V_H^k\rightarrow V_h^k$, the standard prologation operator,
  preserves the divergence with respect to $\tilde{Q}_H$, i.e.\ 
  \begin{equation}
    (\divergence P_H(\veltest_H), \tilde{\presstest}_H) = 
    (\divergence\veltest_H, \tilde{\presstest}_H)\quad
    \text{for all }\tilde{q}_H\in\tilde{Q}_H,~\veltest_H\in V_H^k.
  \end{equation}
\end{itemize}
\end{lemma}
\begin{proof}
To show (a), note that any $q_h\in Q_h^{k-1}$ can be decomposed as
\begin{equation*}
    q_h
    = \sum_{K\in\mathcal{T}_H} \left(q_h-\frac{1}{|K|}\int_K q_h d\boldvar x\right)\chi_K
    +\sum_{K\in\mathcal{T}_H} \left(\frac{1}{|K|}\int_K q_h d\boldvar x\right)\chi_K,
\end{equation*}
where $\chi_K$ is the indicator function for $K$. This shows that
$Q_h^{k-1} = \hat{Q}_h \bigoplus \tilde{Q}_H$.

Next, since $\hat{\veltest}_h=0$ on $\partial K$ for all $K\in \mathcal{T}_H$,
the divergence theorem implies (b).
Since $V_h^k\times Q_h^{k-1}$ is inf-sup stable for the discretization of
\eqref{eq:weakstokes} on each coarse element $K\in\mathcal{T}_H$, i.e.
	\begin{equation*}
		\inf_{q_h\in Q_h^{k-1}}\sup_{\veltest_h\in V_h^k}
		\frac{(q_h, \divergence \veltest_h)_K}
			 {\|\veltest_h\|_{H^1(K)}\|\presstest\|_{L^2(K)}}\geq c_1>0,
	\end{equation*}
by definition of $\hat{V}_h$ and $\hat{Q}_h\subset Q_h^{k-1}$, 
\eqref{eq:hatsp} holds. 
Finally, since $V_H^k \subset V_h^k$, the standard prolongation operator $P_H$ 
is the identity on $V_H^k$, i.e., $P_H\veltest_H = \veltest_H$ for $\veltest_H\in V_H^k$. 
Therefore,
	\begin{equation*}
		(\divergence P_H(\veltest_H), \tilde{\presstest}_H) = 
		(\divergence \veltest_H, \tilde{\presstest}_H),
	\end{equation*}
for all $\tilde{q}_H\in\tilde{Q}_H$, which ends the proof.
\end{proof}
\Cref{lemma:P} verifies the assumptions of
\cite[Proposition 3.1]{FarrellMitchellScottWechsung20b}, whose
application shows that $\tilde{P}_H$ satisfies \eqref{eq:conP}.

\section{Numerical results}\label{sec:results}
In this section, we study the convergence of the linear Stokes solver
combining the AL preconditioners $\mathbf{P}_1$ and
$\mathbf{P}_2$ (described in \cref{sec:AL}) and the parameter-robust
multigrid scheme for the (1,1)-block of the augmented system
\eqref{eq:augstokessys} (described in \cref{sec:multigrid}).  After
providing details of the implementation in \cref{sec:implem}, we study
our solver using two test problems. In
\cref{sec:sinker}, we use the multi-sinker linear Stokes benchmark,
which has already been used in previous sections of this paper to
illustrate basic preconditioning properties. In \cref{sec:tensor}, we
use a nonlinear problem with viscoplastic rheology and study the
behavior of the solver for Newton-type linearizations.

\subsection{Algorithms and implementation}\label{sec:implem}
Our numerical experiments are conducted using 
the open source library Firedrake~\cite{Rathgeber2016,Dalcin2011,
KarypisKumar98,Chaco95,Homolya2016,McRae2016,Luporini2016,Homolya2017,
Homolya2017a,Kirby18,Gibson2020}. 
All problems are specified in their weak forms using the Unified Form Language~\cite{Alnaes14}.
For parallel linear algebra, Firedrake
relies on PETSc \cite{BalayGroppMcInnesEtAl99c}.
The block preconditioner \eqref{eq:precond} is
built up using PETSc's field split preconditioner.
For applications of the inverse Schur complement approximation
$\hat{\mathbf S}^{-1}$ in \eqref{eq:precond}, we assemble the block-diagonal matrices
$\mathbf M_p(1/\mu)^{-1}+\gamma \mathbf M_p^{-1}$ and
$\left(1+\gamma\right) \mathbf M_p(1/\mu)^{-1}$ for $\mathbf P_1$ and
$\mathbf P_2$, respectively, and compute the block-diagonal inverses. 
For the inverse of the approximation of (1,1)-block of
the augmented system \eqref{eq:augstokessys}, we apply a full
geometric multigrid (GMG) cycle using the GMG implementation in
Firedrake \cite{Mitchell2016} with the level operators defined
by rediscretizing the PDEs on each level. For the $\gamma$-robust PSC smoother,
we use a custom preconditioner class that extracts the local problems on the star of each vertex from the
global assembled matrix and solves them using (dense) LU factorization.
We apply 5 pre/post-smoothing steps on each level. 
For the $\gamma$-robust transfer operator, we use Firedrake's ability to provide custom transfer operators.
The matrices required for the local problems on each coarse element are again extracted from the global assembled matrix and solved exactly.
Finally, on the coarsest level, we use the parallel direct
sparse solver MUMPS~\cite{MUMPS01,MUMPS02}.
In all experiments, we use the flexible Krylov solver
FGMRES. A schematic view of the full scheme can be seen in Figure~\ref{fig:algo-outline}.

We present results on
quadrilateral meshes, hexahedral meshes (obtained from extrusion of
quadrilateral meshes \cite{Bercea2016}), and tetrahedral meshes.
Firedrake is designed to run in parallel with the maximum number of
MPI processes being the number of mesh elements $K\in
\mathcal{T}_H$ on the coarse mesh.
For hexahedral meshes, the maximum number of MPI
processes is limited by the number of elements in the quadrilateral
mesh the hexahedral mesh is extruded from.
Having a large number of mesh elements is not only required for parallel distribution,
but we also find that in the presence of extreme viscosity
variations it is necessary that the coarse mesh in the multigrid
hierarchy is sufficiently fine to capture the basic structure of the
viscosity. If the coarse mesh is too coarse, the performance of the
multigrid preconditioner degrades.

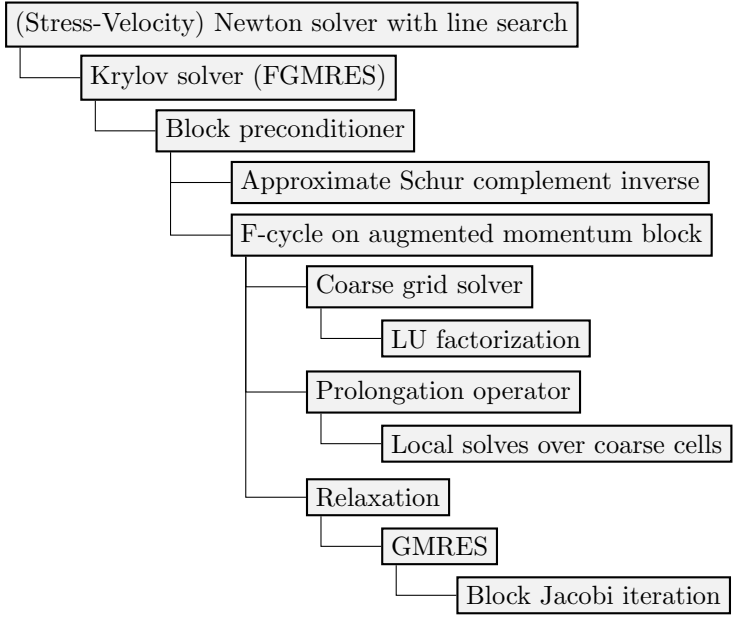
\begin{figure}
        \begin{tikzpicture}[%
        every
        node/.style={draw=black,thick,anchor=west,fill=white!95!black},
        selected/.style={draw=red,fill=red!30},
        optional/.style={dashed,fill=gray!50},
        grow via three points={one child at (1.0,-0.4) and
        two children at (1.0,-0.4) and (1.0,-1.1)},
        edge from parent path={([xshift=2mm] \tikzparentnode.south west) |- (\tikzchildnode.west)},
        growth parent anchor=south west
    ]
    \node {(Stress-Velocity) Newton solver with line search}
        child { node {Krylov solver (FGMRES)}
            child { node {Block preconditioner}
                child { node {Approximate Schur complement inverse}}
                child { node {F-cycle on augmented momentum block}
                    child { node {Coarse grid solver}
                        child { node {LU factorization}}
                    }
                    child [missing] {}
                    child { node {Prolongation operator}
                        child { node {Local solves over coarse cells}}
                    }
                    child [missing] {}
                    child { node {Relaxation}
                        child { node {GMRES}
                            child { node {Block Jacobi iteration}}
                        }
                    }
                }
            }
        };
    \end{tikzpicture}
    \caption{Outline of the full algorithm.}\label{fig:algo-outline}
\end{figure}

The source codes of our implementation are available in a public git
repository\footnote{\url{https://github.com/MelodyShih/vvstokes-al}}.
All experiments are run in parallel on TACC's Frontera or
NYU's Greene system.

\subsection{Multi-sinker problem}\label{sec:sinker}
This is a benchmark problem taken from
\cite{RudiStadlerGhattas17}. The same or
analogous problems have also been used in \cite{MayMoresi08,
MayBrownPourhiet15,BorzacchielloLericheBlottiereGuillet17}. The domain $\Omega$
is a unit square/cube $(0,1)^d$, $d=2,3$ with viscosity $\mu_{\min}>0$. 
Multiple circular/spherical lower viscosity sinkers with diameter
$\omega$ and viscosity $\mu_{\max}>0$ are placed randomly inside the domain. 
The sinker's boundaries are smoothed by a Gaussian kernel with parameter
$\delta$ controlling the smoothness (the lower, the smoother). We denote the 
number of sinkers by $n$ and their centers by $\boldvar c_i$, $i=1,\dots,n$.
The viscosity field $\visc(\boldvar x)\in \mathbb{R}$ is then
specified as
\begin{align*}
\visc(\boldvar x) &:= (\visc_{\max}-\visc_{\min})(1-\chi(\boldvar x)),\quad\boldvar x \in \Omega\\
\chi(\boldvar x) &:=\prod_{i=1}^{n} \Big[1-\exp \big(-\delta \max (0, |\boldvar c_i - \boldvar x| - \omega/2)\big)\Big],
\end{align*}
and the right hand side of \eqref{eq:weakstokes} is
$\boldvar f(\boldvar x):=(0,0,\beta(\chi(\boldvar x)-1)),\enskip \beta=10$,
which forces the sinkers downwards. Homogeneous Dirichlet boundary conditions 
are enforced on the entire boundary $\partial\Omega$. We use the parameters
$\omega=0.1$, $\delta=200$ from \cite{RudiStadlerGhattas17} and fix the
number of sinkers, both in 2D and 3D experiments, to $n=24$. To test the
preconditioner, we vary the dynamic ratio 
$\text{DR}(\visc) := \mu_{\max}/\mu_{\min}$ and assign  
$\mu_{\max} = (\text{DR}(\visc))^{1/2}$ and $\mu_{\min} = (\text{DR}(\visc))^{-1/2}$.

\subsubsection{Influence of AL-parameter $\gamma$ for $[\mathbb{Q}_3]^d\times
          \mathbb{P}_2^{\text{\rm disc}}$ discretization}
\Cref{table:linear} summarizes the convergence behavior of the linear
solver for problems in 2D and 3D. Note that the standard
inverse-viscosity mass matrix Schur complement approximation
($\gamma=0$) requires a large number of iterations or fails to
converge for both, standard geometric multigrid and the parameter-robust
multigrid. This shows the limitations of these Schur complement
approximations for problems with strongly varying viscosity. Next,
note that the number of iterations decreases for larger $\gamma$. This
can be explained by the fact that the block preconditioner
\eqref{eq:block-precond} relies on accurate approximation of both, the
inverse of the (1,1)-block and the Schur
complement. As discussed in \cref{sec:AL}, larger
$\gamma$ improves the Schur complement approximation, but makes the
(1,1)-block of the augmented system more challenging to
solve. However, using the $\gamma$-robust smoother and transfer
operator, the effectiveness of the multigrid scheme for the
(1,1)-block does not degrade for large $\gamma$.
Hence, we
observe a decreasing iteration count as $\gamma$ increases due to the
improved approximation of the Schur complement.
Even when the standard multigrid preconditioner for $\gamma=0$
converges, we observe a shorter computation time for the AL  approach with, e.g., $\gamma=1000$, despite the
computationally more expensive $\gamma$-robust multigrid scheme. That
is, the savings in the number of iterations
overcompensate for the more computational intense and thus slower
$\gamma$-robust smooother and transfer operations.
Last, we note that since the augmented system becomes ill-conditioned
for large
$\gamma$, the value of $\gamma$ cannot be arbitrary large. In practice, we 
found that there is a wide range of $\gamma$ values that leads to
robust convergence. \cref{table:linear} shows that values of $\gamma$ from 
$10$ to $1000$ results in convergence within $30$ iterations in both 2D and 3D
experiments. The same observersation (for an even larger range) can be made from \cref{fig:nonlinear}. 

\begin{table}[ht]
	\label{table:linear}
	\centering
	\caption{Number of FGMRES iterations preconditioned by
          AL preconditioners $\mathbf P_1$, $\mathbf
          P_2$ for 2D and 3D sinker benchmark (left and right table,
          respectively). A dash means that the algorithm was not able
		  to decrease the residual by $10^{-6}$
          within $300$ iterations.  We use $[\mathbb{Q}_3]^d\times
          \mathbb{P}_2^{\text{disc}}$ elements on a quadrilateral mesh
          ($d=2$) and a hexahedral mesh ($d=3$) for velocity and
          pressure. For $\gamma=0$ and the Jacobi smoother and
          standard transfer, both Schur complement preconditioners
          reduce to the inverse viscosity-weighted mass matrix. The
          robust multigrid uses an F-cycle for the (1,1)-block.  For
          the setup of the 2D mesh, FGMRES and multigrid solver, see
          \cref{fig:topleftsolvefail}. For 3D tests, 3 mesh levels are
          used. Number of velocity degrees of freedom is 41,992,563 on
          finest and 680,943 on the coarsest level.}  \hfill
\begin{tabular}{l|rrrrr}
		\toprule
		\multicolumn{1}{l}{\colorbox{black!25}{2D sink.}}&\multicolumn{2}{c}{$\mathbf P_1$} && \multicolumn{2}{c}{$\mathbf P_2$}\\\midrule
		\multicolumn{1}{l}{$\gamma\backslash \text{DR}(\mu)$} & $10^6$ & $10^{10}$ && $10^6$ & $10^{10}$\\\midrule
	    \multicolumn{6}{c}{\small{Jacobi smoother \& standard transfer}}\\\midrule
		0    &55  & - && 55 & -\\\midrule
	    \multicolumn{6}{c}{\small{Robust smoother \& robust transfer}}\\\midrule
		0    &54  & - && 54 &  - \\
		10   &11  & 22  && 19  & 27 \\ 
		1000 &13  & 15  && 12  & 16 \\ 
		\bottomrule
\end{tabular}
\hfill
\begin{tabular}{l|ccc}
		\toprule
		\multicolumn{1}{l}{\colorbox{black!25}{3D sink.}}&$\mathbf P_1$ & $\mathbf P_2$\\\midrule
		\multicolumn{1}{l}{$\gamma\backslash \text{DR}(\mu)$}&$10^8$ & $10^{8}$\\\midrule
	    \multicolumn{3}{c}{\small{Jacobi smoother \& standard transfer}}\\\midrule
	    0    &51  & 51 \\\midrule
	    \multicolumn{3}{c}{\small{Robust smoother \& robust transfer}}\\\midrule
		0    &51  & 51\\
		10   &15  & 15\\ 
		1000 &14  & 14\\ 
		\bottomrule
\end{tabular}
\hfill
\end{table}

\subsubsection{Higher-order discretization}
Next, we examine how the solver performs when we
increase the polynomial order of the
discretization.
\Cref{table:k} summarizes the effect of discretization order on the 
efficiency of the solver.  We test the solver by 
fixing the number of mesh elements.
We find faster convergence for large $\gamma$ 
for all discretization orders. Indeed, \cref{lemma:spectrumest} makes no assumptions on the 
finite element discretizations.
Therefore, one can expect such convergence as long as the (1,1)-block
solver does not degrade as $\gamma$ increases and the choice of the
two matrices $\hat{\mathbf S}$ and $\mathbf W$ (pressure
mass matrix $\mathbf M_p$ and the inverse viscosity weighted pressure
mass matrix $\mathbf M_p(1/\visc)$ in our case) are spectrally
equivalent to the original system's Schur complement.  In particular,
the robustness of the multigrid scheme with respect to $\gamma$ 
holds for higher-order discretizations.

In addition, fixing $\gamma$, we observe a decrease in iteration counts 
as order of discretization grows. The observation has
two reasons. First, there are
more degree of freedoms on the coarsest mesh which can better resolve the viscosity variation when using higher order discretization.  
Second, the PSC smoother we use in the $\gamma$-robust multigrid 
scheme is more powerful for higher order elements: 
recall that the subspace decomposition we found is 
$V_i = \{\veltest_h\in V_h: \text{supp}(\veltest_h) \subset \text{star}(v_i)\}$. 
$V_i$ has dimension $(k-1)^d$, $d=2,3$ (the number of degrees of freedom in the
interior of $\text{star}(v_i)$) which grows as the order $k$ grows. 
Therefore, the operator 
$\mathbf D_{h,\gamma}^{-1} = \sum_i \mathbf I_i \mathbf A_i^{-1}\mathbf I_i^*$
becomes closer to the true inverse $\mathbf A_{h,\gamma}^{-1}$ for
higher order elements. 
We note that this powerful smoother comes at the cost of increased
computational and memory requirements.

\begin{table}[ht]
\centering
\caption{Number of iterations for higher-order $[\mathbb{Q}_k]^3\times
  \mathbb{P}_{k-1}^{\text{disc}}$ element discretizations,
  $k=2,3,4,5$. Results are for the 3D sinker problem with dynamic
  ratio $\text{DR}(\visc)=10^6$ on hexahedral mesh using $\mathbf P_2$
  preconditioner. Only two multigrid mesh levels are used. 
  The number of elements are fixed for all runs ($30$ elements per side of the
  unit cube on the finest level).
  For solver setup, see \cref{fig:topleftsolvefail}.
  }
\label{table:k}
\begin{tabular}{lcccccc}
\toprule
	$\gamma\backslash k$ &&& 2 & 3 & 4 & 5\\\midrule
	\multicolumn{7}{c}{\small{Jacobi smoother \& standard transfer}}\\\midrule
	0    &&&96 &85 & 91& 90 \\\midrule
	\multicolumn{7}{c}{\small{Robust smoother \& robust transfer}}\\\midrule
	0    &&&95 &78 & 79& 82 \\
	10   &&&28 &20 & 20& 21 \\
	1000 &&&27 &13 &  8&  6 \\
\bottomrule
\end{tabular}
\end{table}

\subsubsection{Comparison with monolithic multigrid schemes}
So far, we have focused on comparisons of Schur complement-based Stokes preconditioners. In
this subsection, we examine how the solver compares to a monolithic multigrid 
method, an alternative to Schur complement-based approaches that applies 
multigrid to the to saddle point system directly. 
As mentioned in the introduction, there are different variants of
monolithic multigrid methods. 
We compare with one that uses a Vanka smoother \cite{vanka1986block}, whose
implementation is available with appropriate solver options in PCPATCH 
\cite{FarrellKnepleyMitchellWechsung21}. The scheme is called the
\textit{Full Vanka Smoother} in
\cite{BorzacchielloLericheBlottiereGuillet17} in the context of
a finite volume discretization.  We test the solvers with different
number of sinkers and record the iteration counts in \cref{table:vanka}. 
Both the monolithic multigrid and the AL preconditioner approaches
perform well for a single sinker even with an extreme viscosity
variation. As the number of sinkers grows, the convergence of the monolithic
multigrid scheme slows down for large viscosity variation, whereas the AL 
preconditioner is able to maintain its convergence and only required moderately 
more iterations.

\begin{table}[ht]
    \centering
	\caption{Convergence comparison of FGMRES preconditioned with 
	F-cycle monolithic multigrid scheme with Vanka smoother 
	and preconditioned with the AL preconditioner $\mathbf P_2$ and the robust multigrid scheme. 
	Discretization based on Taylor-Hood elements $[\mathbb Q_2]^2\times \mathbb Q_1$
	and $[\mathbb Q_2]^2\times \mathbb P_1^{\text{disc}}$ for the monolithic multigrid and 
	the AL preconditioner, respectively. 
	Shown are the iteration counts to achieve $10^{6}$ residual reduction. 
	Both methods use 4 mesh levels and 5 pre/post-smoothing steps on each level. 
	The number of unknowns for the velocity is 132,098 for both
        method. The pressure unknowns are
	16,641 for $\mathbb Q_1$ elements and 49,152 for $\mathbb P_1^{\text{disc}}$
	elements. ``-'' indicates failure of the solver to converge in 300 
	iterations.}
	\label{table:vanka}
    \begin{tabular}{cccccccccc}
	\toprule
	&\multicolumn{4}{c}{Mono.\ MG with Vanka } && \multicolumn{4}{c}{AL precond. $\mathbf P_2$ ($\gamma=100$) }\\[0.2cm]\
	\#sinkers  $\backslash \text{DR}(\mu)$ & $10^4$ & $10^6$ & $10^8$ & $10^{10}$ &&  $10^4$ & $10^6$ & $10^8$ & $10^{10}$\\\midrule
	1    & 5 &  6 &  7 & 19  &&  4 &  6 &  8 & 9\\
	6    & 9 & 13 & 26 & -   && 10 & 14 & 16 & 16\\
	24   & 7 & 16 & 85 & -   &&  9 & 15 & 20 & 25\\
	\bottomrule
\end{tabular}
\end{table}

\subsubsection{Mesh refinement and parallel scalability}
For the parallel scalability experiments, we switch to tetrahedral meshes
and the $[\mathbb P_2\oplus B_3^F]^3\times\mathbb P_{0}$
discretization due to the limitation when using hexahedral
meshes in Firedrake discussed in \cref{sec:implem}.
In the table in \cref{fig:weakscal}, 
we verify that also for these meshes, fewer iterations are needed as $\gamma$ increases.
Then, we study the effect of mesh refinement
and the solver's weak parallel scalability. We observe that 
when the mesh is fine enough to sufficiently resolve the viscosity
variations,
the number of iterations 
becomes mesh-independent (\cref{fig:weakscal}, \textit{left}).
To examine the implementation scalability, we focus on the time of the
customized multigrid solve (over 80\% of the total solution time) and normalize
the time by the number of iterations (\cref{fig:weakscal}, \textit{right}). 
The multigrid solver maintains about 96\% parallel efficiency for weak
scalability on 3,584 cores comparing to 56 cores.  
In addition to increased communication costs (in particular for the coarse grid
solve), one reason that the solver slows down for the largest run is load
imbalance.
We note that the complexity of much of the code scales either with the number
of vertices (e.g.~the smoother) or the number of mesh elements (e.g.~assembly
and prolongation).  On the 512 nodes, the maximum number of vertices and 
mesh elements among MPI processes 
is $112\%$ and $63\%$ more than the average number, respectively.
For comparison, on 64 nodes the imbalance is only $47\%$ and $2\%$ respectively.
\begin{figure}[!htb]
\label{fig:weakscal}
\centering
\begin{center}
\begin{tabular}{ccrrr}
	\toprule
	\textcolor{gray}{$l=2$} &$\text{DR}(\mu)\backslash\gamma$ &  \phantom{000}0 & \phantom{00}10 &1000\\\midrule
	$\mathbf P_1$ & $10^6$ & 60 & 35 & 36\\
	\bottomrule
\end{tabular}
\end{center}
\vspace{0.5cm}
\definecolor{clr1}{RGB}{255, 246, 39}
\definecolor{clr2}{RGB}{124, 22, 28}
\definecolor{clr3}{RGB}{84, 170, 25}
\definecolor{clr4}{RGB}{137, 230, 251}
\definecolor{clr5}{RGB}{255, 122, 0} 
\definecolor{clr6}{RGB}{12, 59, 136}
\definecolor{clr7}{RGB}{53, 120, 120}
\definecolor{clr8}{RGB}{50, 49, 70}
\definecolor{clr9}{RGB}{255, 0, 255}
\definecolor{clr10}{RGB}{0, 122, 255}
\definecolor{clr11}{RGB}{255, 122, 122}
\definecolor{clr12}{RGB}{130, 130, 130}
\definecolor{clr13}{RGB}{180, 180, 180}

\hfill
\begin{tikzpicture}[scale=0.6] 
  \begin{semilogxaxis}[
    xtick=data,
    xticklabels={56 \textcolor{gray}{($1$)}, 
	             448 \textcolor{gray}{($2$)}, 
				 3584 \textcolor{gray}{($3$)}, 
				 28672 \textcolor{gray}{($4$)}},
    ybar stacked,
    legend style={
      draw=none,
      at={(0.6,0.98)},
      anchor=north,
      legend columns=3,
      /tikz/every even column/.append style={column sep=0.5cm}},
      legend entries={\bf Total time},
    ymin=0,
    bar width=20pt,
    width=3.5in,
    height=2.4in]
    \addplot[ybar,fill=black!10]
    table[x=cs,y=total] 
	  {data/weak_frontera_result_detail};
  \end{semilogxaxis}
  \begin{semilogxaxis}[
	axis y line*=left,
	xticklabels={},
	yticklabels={},
	xticklabel style={align=center},
    legend style={
      draw=none,
      at={(0.25,0.98)},
      anchor=north,
	  legend entries={\bf Iteration}},
  xlabel=CPU cores $\rightarrow$,
	  xlabel=CPU cores \textcolor{gray}{($l$)} $\rightarrow$,
	  ymajorgrids,ylabel=wall-time (sec.)$\rightarrow$,
    ymin=0,
    ymax=50,
    width=3.5in,
    height=2.4in]
	\addplot[point meta=explicit,nodes near coords
	  align={above},
	  nodes near coords=\pgfmathprintnumber{\pgfplotspointmeta},
	  color=black,mark=square*] 
	  table[meta=it,x=cs,y=it]
	  {data/weak_frontera_result_detail};
  \end{semilogxaxis}
\end{tikzpicture} 
\hfill
\begin{tikzpicture}[scale=0.6] 
  \begin{semilogxaxis}[
    xtick=data,
    xticklabels={56 \textcolor{gray}{($1$)}, 
	             448 \textcolor{gray}{($2$)}, 
				 3584 \textcolor{gray}{($3$)}, 
				 28672 \textcolor{gray}{($4$)}},
	xticklabel style={align=center},
    ybar stacked,
	legend cell align={left},
    legend style={
      draw=none,
      at={(-0.1,1.3)},
      anchor=north west,
      legend columns=3,
      /tikz/every even column/.append style={column sep=0.5cm}},
  legend entries={\bf Level 0 (Coarse), \bf Level $l$, \bf Level $l-1$, \bf
      Level $l-2$, \bf Level $l-3$,
	  \bf Level 0 (Coarse)},
	  xlabel=CPU cores \textcolor{gray}{($l$)} $\rightarrow$,
	  ymajorgrids,ylabel=wall-time (sec.) per iteration $\rightarrow$,
    ymin=0,
    ymax=4,
    bar width=20pt,
    width=3.5in,
    height=2.4in]
    \addplot[color=black, fill=clr5]
    table[x=cs,y expr={(\thisrow{mglevel0})/\thisrow{it}}]
	  {data/weak_frontera_result_detail};

    \addplot[color=black, fill=clr10]
    table[x=cs,y expr={(\thisrow{mglevel4})/\thisrow{it}}] 
	  {data/weak_frontera_result_detail};

    \addplot[color=black, fill=clr10!80!black]
    table[x=cs,y expr={(\thisrow{mglevel3})/\thisrow{it}}] 
	  {data/weak_frontera_result_detail};

    \addplot[color=black, fill=clr10!40!black]
    table[x=cs,y expr={(\thisrow{mglevel2})/\thisrow{it}}]
	  {data/weak_frontera_result_detail};
    \addplot[color=black,
	point meta=explicit,
	nodes near coords=\pgfmathprintnumber{\pgfplotspointmeta}\%,
    nodes near coords align={vertical},
	  fill=clr10!10!black]
    table[meta=eff,x=cs,y expr={(\thisrow{mglevel1})/\thisrow{it}}]
	 {data/weak_frontera_result_detail};
  \end{semilogxaxis}
\end{tikzpicture} 
\hfill
	\caption{Weak scalability results on TACC's
	Frontera (Intel CLX nodes) for the multi-sinker problem using $\mathbf P_1$
	with parameter-robust multigrid solver for the (1,1)-block. 
	The discretization is based on $[\mathbb P_2\oplus B_3^F]^3\times \mathbb
    P_{0}$ elements on tetrahedral mesh. 
	In the table, we show the 
	number of FGMRES iterations to achieve $10^6$ residual reduction with
	different value of $\gamma$ using three mesh levels with
	24,915,603 dofs. In the scalability tests, we use $\gamma=10$. 
	The coarsest mesh has 402,735 dofs for all
	runs. The problem size is increased to maintain about 55K unknowns per
	core.  With $l$ we indicate the number of refinement levels of the finest
	mesh compared to the coarse mesh. On the left we show the iteration numbers
	and the total run times, and on the right the parallel efficiency of one
	multigrid cycle compared to the cycle on $56$ cores. The largest run has
	overall 1.6B unknowns. For problem setup, see \cref{fig:topleftsolvefail}.}
\end{figure}
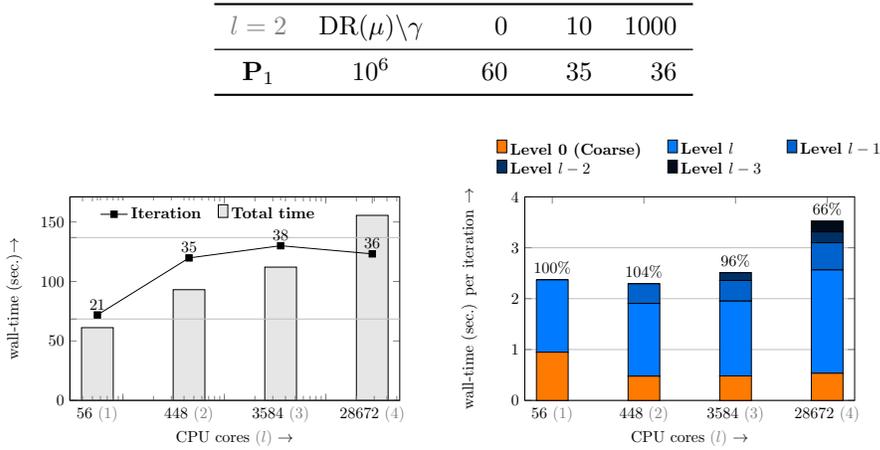

\subsection{Nonlinear Stokes flow with viscoplastic rheology}\label{sec:tensor}
So far, we have used the solver for linear Stokes equations with
scalar, strongly spatially varying viscosity. In this section, we
examine the solver for nonlinear problems where, upon linearization,
the viscosity field $\mu$ is an anisotropic fourth-order tensor.

\subsubsection{Linearization}\label{subsec:viscplast-Newton}
We apply the solver to Newton linearizations of nonlinear Stokes flow with a
viscoplastic rheology, i.e.,
$$
	\visc(\boldvar x,\strainrateinvII)
	=\dfrac{\viscref\viscstressyield}{2\viscref\strainrateinvII+
	\viscstressyield},
$$
where $\viscref>0$ is a reference viscosity, and $\viscstressyield>0$ is a
given yield stress. We refer to 
$\visc(\boldvar x,\strainrateinvII)$ 
as effective viscosity. Fluids with this rheology have two
fundamentally different behavior regimes. For small
$\strainrateinvII$, i.e., in the viscous regime, they behave like a Newtonian fluid with constant viscosity $\viscref$
. In the plastic regime,
i.e., for large $\strainrateinvII$, the effective viscosity becomes
small such that the second invariant of the stress is bounded by
$\viscstressyield$. Such fluids occur, for instance, in
the geosciences \cite{SpiegelmanMayWilson16,Ranalli95}.
We use the stress--velocity
Newton linearizations from \cite{RudiShihStadler20}, which
leads, in the $k$-th iteration, to the linear Stokes system for the
Newton increment variables $(\tilde\vel,\tilde p)$,
$$
\begin{aligned}
-\nabla\cdot
\left[\dfrac{2\viscref\viscstressyield}{2\viscref\strainrateinvII^{\vel_{k-1}}+\viscstressyield}
\bigg(\mathbb I -
\dfrac{\left(\strainratetensor(\vel_{k-1})\otimes \viscstresstensor_{k-1}\right)_{\text{sym}}}
{2\strainrateinvII^{\vel_{k-1}}\max(\viscstressyield,\viscstressinvII)}
\bigg)\strainratetensor(\tilde{\vel})\right]
+ \nabla \tilde p &= -\boldvar r^{\vel}_{k-1}
\\
-\nabla\cdot\tilde{\boldvar u} &= -r^{p}_{k-1}.
\end{aligned}
$$
Here, $\viscstresstensor_{k-1}$ is the independent variable for the
viscous stress tensor that is introduced in the stress--velocity
Newton method, $\boldvar r^{\vel}_{k-1}$ and $r^{p}_{k-1}$ are
residuals, $\mathbb I$ denotes the identity tensor, and $\otimes$
denoting the outer product between two second-order tensors. Details
of this stress--velocity Newton method and an update formula for
$\viscstresstensor$ can be found in \cite{RudiShihStadler20}, where it
is also shown that compared to a standard Newton linearization, this
alternative linearization improves nonlinear convergence. Note
that a standard Newton method requires solution of a very similar system, with
the main difference being that $\viscstresstensor_{k-1}$ is replaced by
$\strainratetensor(\vel_{k-1})$.

\subsubsection{Problem setup}
The domain $\Omega$ is a $120$ km
$\times$ $7.5$ km $\times$ $30$ km rectangular box which has a viscoplastic
lower layer with reference viscosity 
$\visc_1$ and yield stress $\viscstressyield$, and a constant viscosity
upper layer with viscosity $\visc_2$. There is a notch-like domain
introduced in the lower layer with constant viscosity $\visc_3$. 
The geometry is identical in the $y$ direction. At the left and
right sides, we prescribe inflow boundary conditions, $\boldvar u (x,y,z)\cdot 
\boldvar{n} = -u_0(1+y)$,
and at the fore, aft, and bottom boundaries we use $\boldvar u (x,y,z)\cdot 
\boldvar{n} = 0$. At the top and for tangential velocities, we use homogeneous 
Neumann boundary conditions, and $\boldvar f \equiv \boldvar 0$. 
We use the parameters $u_0 = 2.5~\text{mm/yr}$, $\mu_1 =
10^{24}~\text{Pa}\:\text{s}$, $\mu_2 = 10^{21}~\text{Pa}\:\text{s}$, 
$\mu_3 = 10^{17}~\text{Pa}\:\text{s}$ and $\viscmin = 10^{15}~\text{Pa}\:\text{s}$
and nondimensionalized them
by $H_0 = 30$ km, $U_0 = 2.5\times10^{-3}\text{ (m/year)}\times
1/3600/365.25/24 \text{ (year/s) }$ and $\eta_0=10^{21} \text{Pa}\cdot\text{s}$.
The mesh in the $x\times z$-plane is constructed using an unstructured quadrilateral mesh 
using Gmsh\cite{GeuzaineRemacle09}, which is then extruded
it the $y$-direction using Firedrake. 
The mesh resolves the boundary between the notch-like domain and the boundary
between the upper and lower layers.

To set up the AL preconditioner, we use the scalar quantity
in front of the fourth order tensor, i.e., the effective viscosity at
iteration $k-1$, i.e.,
\begin{equation*}
\visc(\boldvar x, \strainrateinvII^{\vel_{k-1}}) =\dfrac{\viscref\viscstressyield}{2\viscref\strainrateinvII^{\vel_{k-1}}+\viscstressyield}
\end{equation*}
to compute the inverse viscosity-weighted pressure
mass matrix.
In \cref{fig:3D}, we show the effective viscosity and the second invariant 
of the strain rate tensor for the solution of the nonlinear
problem. The high strain rate shear bands occur dynamically due to the
nonlinearity of the rheology. At
convergence, the effective viscosity field varies over seven orders of
magnitude.

\begin{figure}
  \begin{center}
    \includegraphics[width=0.45\textwidth]{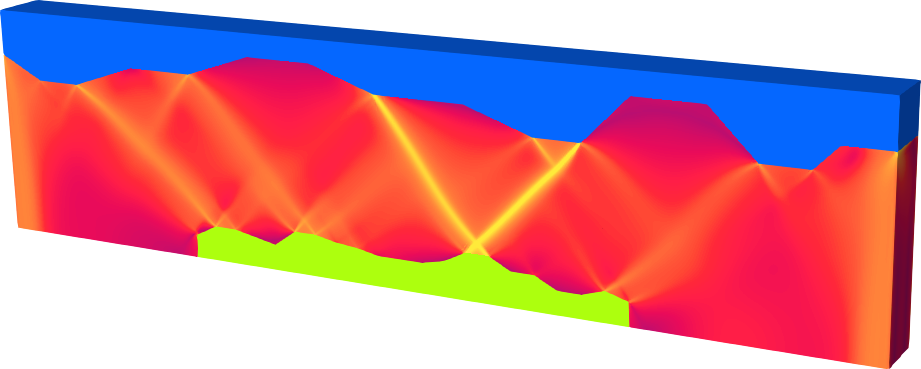}
	\hspace{.5cm}
    \includegraphics[width=0.45\textwidth]{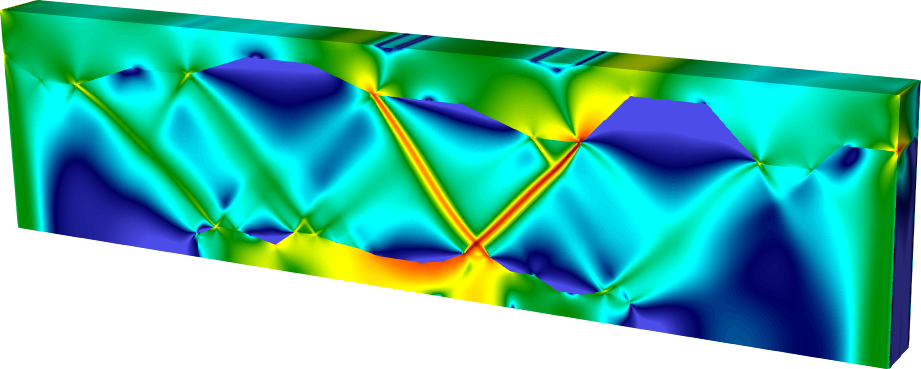}
	\\
	\vspace{-0.3cm}\hspace{-1cm}\includegraphics[width=0.3\textwidth]{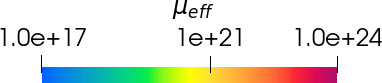}
	\hspace{2cm}
	\includegraphics[width=0.3\textwidth]{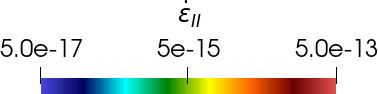}
	\hspace{1cm}
  \end{center}
	\caption{Effective viscosity (\textit{left}) and the second invariant of
	the strain rate tensor (\textit{right}) pertaining to the 3D compressional 
	problem with composite rheology described in \cref{sec:tensor}. 
	}
  \label{fig:3D}
\end{figure}

\def\datafold{data/nonlinear/rate1}
\begin{figure}
   \centering
	\begin{tikzpicture}[scale=1]
   \centering
   \begin{axis}[title={},
   grid = major,
   ymode = log,
   width=9.5cm,
   height=5.5cm,
   xlabel=nonlinear iteration, 
   ylabel=\# linear iteration, 
   legend pos=outer north east,
   legend image post style={mark indices={}},
   xmin=0.5, xmax=14.5, legend cell align={left}]

   \addplot[mark=none, black, dashed, domain=0:15,forget plot] {300};

   \addplot[blue!80!black,mark=x,mark size=3pt, only marks,line width=1.5pt] table {\datafold/S_stressvel_comp_vo3_gamma0.out};
   \addlegendentry{$\gamma=0$, Failed}
  
   \addplot[purple!80!black,mark=*, mark size=1.5pt,mark indices={1,2,3}] table {\datafold/S_stressvel_comp_vo3_gamma01.out};
	   \addlegendentry{$\gamma=0.1$}

   \addplot[only marks,purple!80!black,mark=x, mark size=3pt,mark indices={4},line width=1.5pt] table {\datafold/S_stressvel_comp_vo3_gamma01.out};
   \addlegendentry{$\gamma=0.1$, Failed}

   \addplot[orange,mark=square*,mark size=1.5pt] table {\datafold/S_stressvel_comp_vo3_gamma1.out};
	   \addlegendentry{$\gamma=1$ (27)}

   \addplot[green!80!black,mark=triangle*,mark size=2.5pt] table {\datafold/S_stressvel_comp_vo3_gamma10.out};
	   \addlegendentry{$\gamma=10$ (29)}
   
   \addplot[yellow!80!black,mark=diamond*,mark size=2.5pt] table {\datafold/S_stressvel_comp_vo3_gamma100.out};
	   \addlegendentry{$\gamma=100$ (41)}

   \addplot[red!80!black,mark=star, mark size=2pt,mark options={line width=1pt}] table {\datafold/S_stressvel_comp_vo3_gamma1000.out};
	   \addlegendentry{$\gamma=1000$ (45)}
   \end{axis}
   \end{tikzpicture}
   \caption{Number of iterations for solving the linearized
     Stokes systems ($y$-axis) in each Newton linearization ($x$-axis)
	 for viscoplastic rheology problem. 
	The average number of iterations for the linearized system are
    shown in parenthesis in the legend. 
	The stress--velocity Newton solver is run until reaching
	$10^8$ nonlinear residual reduction.
	The discretization is based on 
	$[\mathbb{Q}_3]^3\times\mathbb{P}_{2}^{\text{disc}}$ element with 25,122,399
	unknowns. The multigrid hierarchy has 3 mesh levels. 
	For solver settings (FGMRES, multigrid), 
	see \cref{fig:topleftsolvefail}.
	}
   \label{fig:nonlinear}
\end{figure}
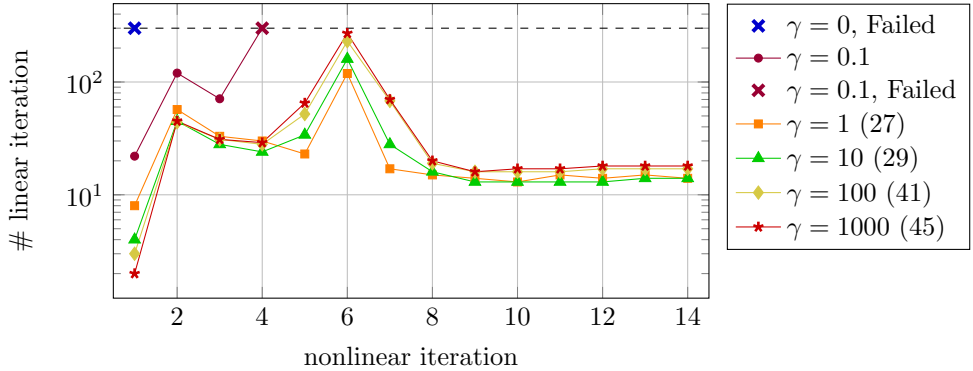

\subsubsection{Linear and nonlinear convergence}
\cref{fig:nonlinear} shows the convergence history for the nonlinear problem.
First, we observe that using the AL preconditioner is necessary 
for the problem. 
For $\gamma=0$, i.e., the inverse viscosity-weighted pressure mass
matrix as the Schur complement approximation in
\eqref{eq:block-precond}, the linear solver
fails to solve the first stress--velocity Newton linearization 
within $300$ iterations.
Second, the number of linear iterations required in each Newton step varies.
For instance, the linearization arising in the $6$-th nonlinear iteration seems
particularly difficult to solve. This happens as the linearized systems in different
nonlinear iterations may have different characteristics and some may
be more difficult to solve than others.  On average, with large enough
$\gamma$ ($\gamma > 0.1$), the AL solver requires  between 27
and 45 linear iterations.  Lastly, we note that while the speed of
convergence depends on the choice of $\gamma$, robust convergence is
observed for a wide range of $\gamma$'s. In particular, 
values of $\gamma$ from $1$ to $1000$ result in an
average of under 50 iterations per linear solve for a problem that
could not be solved without the AL preconditioner.

\subsubsection{Parallel scalability}
Figure \ref{fig:weakscalnonlinear} shows the parallel
scalability of the solver when applying to the nonlinear problem
on hexahedral meshes with the $[\mathbb{Q}_2]^3\times
\mathbb{P}_1^{\text{disc}}$ discretization.  We note that as of writing,
Firedrake only supports hexahedral meshes via
extrusion of quadrilateral meshes, which has the consequence that only the two
dimensional base mesh can be distributed in parallel.  This limits the number
of cores that can be used and makes the distribution over a large number of
cores more challenging than in the case of tetrahedral meshes.
For the slab domain under consideration, we are still able to scale to 1536
cores and 151 million unknowns, with a parallel efficiency of about 57\%
percent compared to a small run on 24 cores.  We note that for the largest run,
the maximum number of vertices and mesh elements on the finest level among MPI
processes compared to the average number is $40\%$ and $18\%$ larger,
respectively.  A fourth run was not possible as then the number of cores would
have exceeded the number of quadrilaterals in the base mesh.
\begin{figure}[!htb]
\label{fig:weakscalnonlinear}
\centering
\definecolor{clr1}{RGB}{255, 246, 39}
\definecolor{clr2}{RGB}{124, 22, 28}
\definecolor{clr3}{RGB}{84, 170, 25}
\definecolor{clr4}{RGB}{137, 230, 251}
\definecolor{clr5}{RGB}{255, 122, 0} 
\definecolor{clr6}{RGB}{12, 59, 136}
\definecolor{clr7}{RGB}{53, 120, 120}
\definecolor{clr8}{RGB}{50, 49, 70}
\definecolor{clr9}{RGB}{255, 0, 255}
\definecolor{clr10}{RGB}{0, 122, 255}
\definecolor{clr11}{RGB}{255, 122, 122}
\definecolor{clr12}{RGB}{130, 130, 130}
\definecolor{clr13}{RGB}{180, 180, 180}

\hfill
\begin{tikzpicture}[scale=0.6] 
  \begin{semilogxaxis}[
    xtick=data,
    xticklabels={24 \textcolor{gray}{($1$)}, 
	             192 \textcolor{gray}{($2$)}, 
				 1536 \textcolor{gray}{($3$)}},
    ybar stacked,
    legend style={
      draw=none,
      at={(0.6,0.98)},
      anchor=north,
      legend columns=2,
      /tikz/every even column/.append style={column sep=0.5cm}},
      legend entries={\bf Total time},
    ymin=0,
    ymax=100,
    bar width=20pt,
    width=3.5in,
    height=2.4in]
    \addplot[ybar,fill=black!10]
    table[x=ns,y=total] 
	  {data/nonlinear_weak_greene_result_detail};
  \end{semilogxaxis}
  \begin{semilogxaxis}[
	axis y line*=left,
	xticklabels={},
	yticklabels={},
	xticklabel style={align=center},
    legend style={
      draw=none,
      at={(0.25,0.98)},
      anchor=north,
	  legend entries={\bf Iteration}},
  xlabel=CPU cores $\rightarrow$,
	  xlabel=CPU cores \textcolor{gray}{($l$)} $\rightarrow$,
	  ymajorgrids,ylabel=wall-time (sec.)$\rightarrow$,
    ymin=0,
    ymax=12,
    width=3.5in,
    height=2.4in]
	\addplot[nodes near coords
	  align={above},
	  nodes near coords=\pgfmathprintnumber{\pgfplotspointmeta},
	  color=black,mark=square*] 
	  table[x=ns,y=it]
	  {data/nonlinear_weak_greene_result_detail};
  \end{semilogxaxis}
\end{tikzpicture} 
\hfill
\begin{tikzpicture}[scale=0.6] 
  \begin{semilogxaxis}[
    xtick=data,
    xticklabels={24 \textcolor{gray}{($1$)}, 
	             192 \textcolor{gray}{($2$)}, 
				 1536 \textcolor{gray}{($3$)}},
	xticklabel style={align=center},
    ybar stacked,
	legend cell align={left},
    legend style={
      draw=none,
      at={(0.1,1.3)},
      anchor=north west,
      legend columns=2,
      /tikz/every even column/.append style={column sep=0.5cm}},
  legend entries={\bf Level 0 (Coarse), \bf Level $l$, \bf Level $l-1$, \bf
      Level $l-2$, \bf Level $l-3$,
	  \bf Level 0 (Coarse)},
	  xlabel=CPU cores \textcolor{gray}{($l$)} $\rightarrow$,
	  ymajorgrids,ylabel=wall-time (sec.) per iteration$\rightarrow$,
    ymin=0,
    ymax=7,
    bar width=20pt,
    width=3.5in,
    height=2.4in]
    \addplot[color=black, fill=clr5]
    table[x=ns,y expr={(\thisrow{mglevel0})/\thisrow{it}}]
      {data/nonlinear_weak_greene_result_detail};

    \addplot[color=black, fill=clr10]
    table[x=ns,y expr={(\thisrow{mglevel1})/\thisrow{it}}] 
      {data/nonlinear_weak_greene_result_detail};

    \addplot[color=black, fill=clr10!80!black]
    table[x=ns,y expr={(\thisrow{mglevel2})/\thisrow{it}}] 
      {data/nonlinear_weak_greene_result_detail};


    \addplot[color=black,
	  point meta=explicit,
	  nodes near coords=\pgfmathprintnumber{\pgfplotspointmeta}\%,
      nodes near coords align={vertical},
	  fill=clr10!10!black]
    table[meta=eff,x=ns,y expr={(\thisrow{mglevel3})/\thisrow{it}}]
	 {data/nonlinear_weak_greene_result_detail};
  \end{semilogxaxis}
\end{tikzpicture} 
\hfill
\caption{Weak scalability results on NYU Greene for 
  the first linear solve of the nonlinear Stokes problem 
  with $\mathbf P_2$ with parameter robust
  multigrid solver for the (1,1)-block. 
  The coarsest mesh is the same across the runs and the problem size is
  increased to maintain approximately 100k unknowns per core.
  With $l$ we indicate the number of refinement levels of the finest mesh compared to the coarse mesh.
  On the left we show the iterations numbers and the total
  times, and on the right the parallel efficiency of one multigrid
  cycle compared to the cycle on $24$ cores. The largest run has
  151 million unknowns.}
\end{figure}
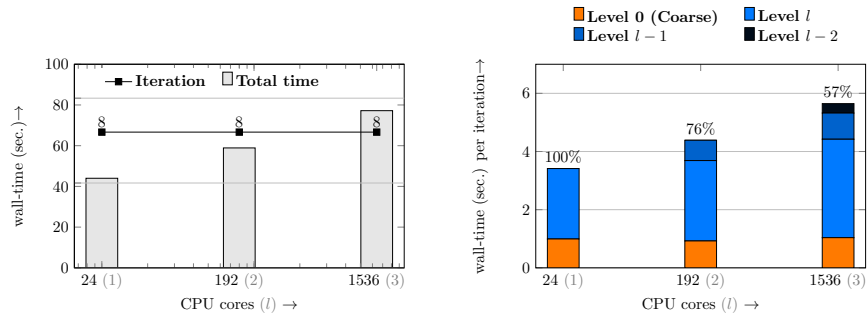

\section{Conclusions}
In this work we developed a scalable preconditioner for the Stokes equations
with varying viscosity.  The preconditioner combines an augmented Lagrangian
term, a mass matrix based Schur complement approximation, and a robust
multigrid scheme for the resulting nearly singular (1,1)-block.  The two main
contributions are eigenvalue estimate for the Schur complement approximation as
well as a multigrid scheme for the (1,1)-block for the popular
$[\mathbb{Q}_k]^d\times\mathbb{P}_{k-1}^{\text{disc}}$ discretization on
quadrilateral/hexahedral meshes.  Numerical experiments confirm robustness even
for large viscosity contrasts, scalability to large problems in three
dimensions, and show that the preconditioner can be combined with the
stress-velocity Newton method of~\cite{RudiShihStadler20} to solve nonlinear
Stokes flow with viscoplastic rheology.  Finally, we remark that we expect that
the approach here can be used for the development of preconditioners for the
Navier-Stokes equations, in the same way that the multigrid scheme developed
in~\cite{FarrellMitchellScottWechsung20b} yields the Reynolds-robust
preconditioner for the Navier--Stokes equations on simplicial meshes developed
in~\cite{FarrellMitchellScottWechsung20a}.

\section*{Acknowledgments}
We appreciate many helpful discussions about the Firedrake project
with Lawrence Mitchell.  Our simulations used the Greene HPC system at
NYU as well as the Frontera computing project at the Texas Advanced
Computing Center.  Frontera is made possible by National Science
Foundation award OAC-1818253.

\bibliographystyle{siamplain}
\bibliography{references,ccgo,firedrake}
\end{document}